\documentclass[hidelinks]{amsart}
\usepackage{mathtools}
\usepackage{amssymb}
\usepackage{xparse}
\usepackage{xr-hyper}
\usepackage{array}

\usepackage{bm,bbm}
\usepackage{mathrsfs}
\usepackage{stmaryrd}
\usepackage[scr=boondoxo]{mathalfa}

\usepackage[centering, footskip=6mm]{geometry}

\usepackage{iftex}

\ifpdf
\usepackage[hypertexnames=false]{hyperref}
\else
\usepackage[dvipdfmx,hypertexnames=false]{hyperref}
\fi

\usepackage[nameinlink]{cleveref}

\usepackage{etoolbox}
\makeatletter
\patchcmd\@thm
 {\let\thm@indent\indent}{\let\thm@indent\noindent}%
 {}{}
\expandafter\patchcmd\csname\string\proof\endcsname
  {\normalparindent}{0pt }{}{}
\makeatother
\numberwithin{equation}{section}

\makeatletter
\renewcommand\th@plain{\slshape}
\makeatother

\newtheoremstyle{plain}
  {}
  {}
  {\slshape}
  {}
  {\sffamily\bfseries}
  {.}
  {.5em}
  {}
\theoremstyle{plain}
\newtheorem{theorem}{Theorem}
\newtheorem{pretheorem}{Theorem}

\newtheorem{corollary}{Corollary}
\newtheorem{lemma}{Lemma}

\newtheoremstyle{definition}
  {}
  {}
  {\normalfont}
  {}
  {\sffamily\bfseries}
  {.}
  {.5em}
  {}
\theoremstyle{definition}

\newtheorem{definition}{Definition}

\newtheorem{memo}{Memo}

\newtheorem{example}{Example}

\crefname{section}{Section}{Sections}

\crefname{theorem}{Theorem}{Theorems}
\crefname{pretheorem}{Theorem}{Theorems}
\crefname{corollary}{Corollary}{Corollaries}
\crefname{lemma}{Lemma}{Lemmas}
\crefname{proposition}{Proposition}{Propositions}
\crefname{claim}{Claim}{Claims}
\crefname{definition}{Definition}{Definitions}
\crefname{notation}{Notation}{Notations}
\crefname{problem}{Problem}{Problems}
\crefname{question}{Question}{Questions}
\crefname{note}{Note}{Notes}
\crefname{memo}{Memo}{Memos}
\crefname{remark}{Remark}{Remarks}
\crefname{example}{Example}{Examples}

\crefname{enumi}{}{}
\crefname{enumii}{}{}
\crefname{enumiii}{}{}

\crefformat{equation}{{\upshape(#2#1#3)}}

\expandafter\def\csname ver@etex.sty\endcsname{3000/12/31}

\usepackage{autonum}

\makeatletter
\newcommand{\restore@Environment}[1]{%
  \AtBeginDocument{%
    \csletcs{#1*}{#1}%
    \csletcs{end#1*}{end#1}%
  }%
}
\forcsvlist\restore@Environment{alignat,equation,gather,multline,flalign,align}
\makeatother

\usepackage{enumitem}
\setlist{leftmargin=20pt}
\setlist[enumerate]{label=\textup{(\roman*)}, leftmargin=*}

\ifpdf
\usepackage{graphicx}
\else
\usepackage[dvipdfmx]{graphicx}
\fi
\usepackage{tikz}
\usetikzlibrary{arrows.meta}
\usetikzlibrary{tikzmark}
\usetikzlibrary{cd}
\tikzcdset{ampersand replacement=\&}
\tikzcdset{
  cells={font=\everymath\expandafter{\the\everymath\displaystyle}},
}

\allowdisplaybreaks

\newif\ifshowkeys

\showkeysfalse

\ifshowkeys
\usepackage{showkeys}

\makeatletter
  \SK@def\cref#1{\SK@\SK@@ref{#1}\SK@cref{#1}}%
\makeatother
\fi

\let\tmp\phi
\let\phi\varphi
\let\varphi\tmp
\let\tmp\epsilon
\let\epsilon\varepsilon
\let\varepsilon\tmp

\renewcommand{\subset}{\subseteq}

\renewcommand{\emptyset}{\varnothing}

\let\tmp\liminf
\let\liminf\varliminf
\let\varliminf\tmp
\let\tmp\limsup
\let\limsup\varlimsup
\let\varlimsup\tmp

\makeatletter
\NewDocumentCommand{\xsideset}{mmme{_^}}{%
\mathop{%
\settowidth{\dimen0}{$\m@th\displaystyle#3$}%
\dimen0=.5\dimen0
\settowidth{\dimen2}{$%
\m@th\displaystyle#3%
\IfValueT{#4}{_{#4}}%
\IfValueT{#5}{^{#5}}%
$}%
\dimen2=.5\dimen2
\advance\dimen2 -\dimen0
\sbox6{\scriptspace\z@$\displaystyle{\vphantom{#3}}#1$}
\sbox8{\scriptspace\z@$\displaystyle{\vphantom{#3}}#2$}
\ifdim\wd6>\dimen2 \kern\dimexpr\wd6-\dimen2\relax\fi
{%
\mathop{\llap{\copy6}{\displaystyle#3}\rlap{\copy8}}\limits%
\IfValueT{#4}{_{#4}}%
\IfValueT{#5}{^{#5}}%
}%
\ifdim\wd8>\dimen2 \kern\dimexpr\wd8-\dimen2\relax\fi
}%
}
\makeatother

\makeatletter
\newcommand*{\house}[1]{%
  \mathord{%
    \mathpalette\@house{#1}%
  }%
}
\newcommand*{\@house}[2]{%
  \dimen@=\fontdimen8 %
      \ifx#1\scriptscriptstyle\scriptscriptfont
      \else\ifx#1\scriptstyle\scriptfont
      \else\textfont\fi\fi
      3 %
  \sbox0{%
    $#1%
      \vrule width\dimen@\relax
      \overline{%
        \kern2\dimen@
        \begingroup 
          #2%
        \endgroup
        \kern2\dimen@
      }%
      \vrule width\dimen@\relax
      \mathsurround=1.5\dimen@ 
    $%
  }%
  \ht0=\dimexpr\ht0-\dimen@\relax
  \dp0=\dimexpr\dp0+2\dimen@\relax
  \vbox{%
    \kern\dimen@ 
    \copy0 %
  }%
}

\colorlet{Hajime}{yellow!60!black}  
\colorlet{Yuta}{green!60!black}     
\colorlet{Yohei}{blue!60!black}     

\begin{document}

\title[Irrationality criterion for sparse infinite series]
{Refinements of Erd\H{o}s's irrationality criterion\\
for certain sparse infinite series}
\author[H. Kaneko]{Hajime Kaneko}
\address{%
Institute of Mathematics, University of Tsukuba,
1-1-1 Tennodai Tsukuba, Ibaraki 305-8571, Japan}
\email{kanekoha@math.tsukuba.ac.jp}
\author[Y. Suzuki]{Yuta Suzuki}
\address{%
Department of Mathematics, Rikkyo University,
3-34-1 Nishi-Ikebukuro, Toshima-ku, Tokyo 171-8501, Japan}
\email{suzuyu@rikkyo.ac.jp}
\author[Y. Tachiya]{Yohei Tachiya}
\address{%
Graduate School of Science and Technology, Hirosaki University,
3 Bunkyo-cho, Hirosaki, Aomori 036-8561, Japan}
\email{tachiya@hirosaki-u.ac.jp}
\date{\today}
\maketitle

\begin{abstract}
In this paper, we establish new irrationality criteria
for certain sparse power series. 
As applications of these criteria,
we generalize a result of Erd\H{o}s 
and obtain several irrationality results for 
various infinite series involving the classical arithmetic functions. 
For example, we prove that for any integers $t\ge2$ and $k\geq0$, the numbers 
\[
\sum_{n=1}^{\infty}
\frac{d(n)^k}{t^{\sigma(n)}}
\quad\text{and}\quad
\sum_{n=1}^{\infty}
\frac{d(n)^k}{t^{\phi(n)}}
\]
are both irrational, where $d(n)$, $\sigma(n)$, and $\phi(n)$ denote  
the number of divisors, the sum of divisors, 
and Euler's totient functions, respectively. 
\end{abstract}

\section{Introduction and main results}
\label{sec:introduction}
Let $t\geq2$ be an integer and  
$(n_k)_{k=1}^{\infty}$ be a strictly increasing sequence of positive integers.
In 1954, Erd\H{o}s and Straus~\cite{Erdos:ES1954} proved that 
the number $\sum_{k=1}^{\infty}t^{-n_k}$ is transcendental 
if $(n_k)_{k=1}^{\infty}$ satisfies $\varlimsup_{k\to\infty}{\log n_k}/{\log k}=\infty$. 
For example, the numbers
\[
\sum_{n=1}^{\infty}\frac{1}{2^{\lfloor n^{\log n}\rfloor}}
\quad\text{and}\quad 
\sum_{n=2}^{\infty}\frac{1}{2^{\lfloor n^{\log \log n}\rfloor}}
\]
are transcendental, where $\lfloor x \rfloor$ denotes the greatest integer not exceeding $x$.  
Shortly thereafter, Erd\H{o}s~\cite{Erdos:PowerSeriesIrrationalitySigmaPhi}
has improved this result by showing that if 
$(n_k)_{k=1}^{\infty}$ satisfies $\varlimsup_{k\to\infty}{n_k}/{k^{\ell}}=\infty$
for some integer $\ell\ge1$, then the number 
$\sum_{k=1}^{\infty}t^{-n_k}$ is either transcendental or algebraic of degree at least $\ell+1$; 
in particular, the numbers
\[
\sum_{n=1}^{\infty}
\frac{1}{2^{n^3}}
\quad\text{and}\quad
\sum_{p:\text{prime}}
\frac{1}{2^{p^2}}
\] 
are neither rational nor quadratic irrational. 
Erd\H{o}s's proof depends on Lemma~4 of \cite{Erdos:PowerSeriesIrrationalitySigmaPhi},
a criterion for irrationality of certain infinite series. 
As its application, he also derived the irrationality of the numbers 
\begin{equation}\label{sigma_phi}
\sum_{n=1}^{\infty}\frac{1}{t^{\sigma(n)}}
\quad\text{and}\quad
\sum_{n=1}^{\infty}\frac{1}{t^{\phi(n)}},
\end{equation}
where  $\sigma(n)$ and $\phi(n)$  are the sum of divisors 
and Euler's totient functions defined, respectively, by 
\begin{equation}\label{def:sigma_phi}
\sigma(n)
\coloneqq
\sum_{d\mid n}
d
\quad\text{and}\quad 
\phi(n)
\coloneqq
\#\{m\in\mathbb{Z}\mid (m,n)=1,\ 1\le m\le n\}
\quad\text{for $n\ge 1$}.
\end{equation}
Moreover, in the same paper, he introduced the following \cref{prethm:ErdosLemma4prime}  
as a more general version of Lemma~4 of \cite{Erdos:PowerSeriesIrrationalitySigmaPhi}; 
however, the details of the proof have been omitted. 
For a sequence $\bm{a}\coloneqq(a(n))_{n=1}^{\infty}$ of integers and for $x>1$, 
we define 
\begin{equation}
\label{def:N}
\mathcal{N}_{\bm{a}}
\coloneqq
\{
n\in\mathbb{Z}_{\ge1}
\mid
a(n)\neq0\},
\end{equation}
\begin{equation}
\label{def:SNx}
S_{\bm{a}}(x)
\coloneqq
\sum_{n<x}|a(n)|
\quad\text{and}\quad 
\mathcal{N}_{\bm{a}}(x)
\coloneqq
\mathcal{N}_{\bm{a}}\cap[1,x).
\end{equation}
Note that we will extend the definition of $S_{\bm{a}}(x)$
to sequences of algebraic integers at \cref{def:SNx_extended}.

\begin{pretheorem}[{\cite[Lemma~4']{Erdos:PowerSeriesIrrationalitySigmaPhi}}]
\label{prethm:ErdosLemma4prime}
Let $t\ge2$ be an integer and 
\[
\bm{a}\coloneqq(a(n))_{n=1}^{\infty}
\quad\text{and}\quad
\bm{b}\coloneqq(b(n))_{n=1}^{\infty}
\]
be sequences of integers such that $a(n)\ge0$ for $n\ge1$, $\#\mathcal{N}_{\bm{a}}=\infty$ and
\newcounter{prethm:ErdosLemma4prime:i}
\setcounter{prethm:ErdosLemma4prime:i}{0}
\begin{enumerate}[leftmargin=*]
\setcounter{enumi}{\value{prethm:ErdosLemma4prime:i}}
\item\label{prethm:ErdosLemma4prime:pointwise}
$\varlimsup_{n\to\infty}(a(n)+|b(n)|)^{1/n}<t$.
\setcounter{prethm:ErdosLemma4prime:i}{\value{enumi}}
\end{enumerate}
Suppose that there exists a sequence
$(x_{n})_{n\ge 1}$  
of real numbers $\ge1$ satisfying the following conditions as $n\to\infty$\textup{:}
\begin{enumerate}[leftmargin=*]
\setcounter{enumi}{\value{prethm:ErdosLemma4prime:i}}
\item\label{prethm:ErdosLemma4prime:x_infty}
$x_{n}\to\infty$. 
\item\label{prethm:ErdosLemma4prime:moderate_growth}
$S_{\bm{a}}(x_{n}), S_{\bm{b}}(x_{n})=O(x_{n})$.
\item\label{prethm:ErdosLemma4prime:sparsity}
$\#\mathcal{N}_{\bm{a}}(x_{n}),\#\mathcal{N}_{\bm{b}}(x_{n})=o(x_{n})$.
\setcounter{prethm:ErdosLemma4prime:i}{\value{enumi}}
\end{enumerate}
Moreover, if $\#\mathcal{N}_{\bm{b}}=\infty$,
we assume that  there exist constants $\Delta,L>1$ satisfying the condition 
\begin{enumerate}[leftmargin=*]
\setcounter{enumi}{\value{prethm:ErdosLemma4prime:i}}
\item
For any consecutive indices $m<m_{+}$ in $\mathcal{N}_{\bm{b}}$
and any real number $\mu\ge L$ with $m+\Delta\mu<m_{+}$,
we have $\mathcal{N}_{\bm{a}}\cap [m+\mu,m+\Delta\mu)\neq\emptyset$.
\setcounter{prethm:ErdosLemma4prime:i}{\value{enumi}}
\end{enumerate}
Then the number 
\begin{equation}
\label{prethm:ErdosLemma4prime:series}
\sum_{n=1}^{\infty}
\frac{a(n)+b(n)}{t^{n}}
\end{equation}
is irrational. 
\end{pretheorem} 

The aim of the present paper is to give new irrationality criteria
for sparse power series of the type \cref{prethm:ErdosLemma4prime:series}.
As applications,
we will derive a generalization of the above Erd\H{o}s's \cref{prethm:ErdosLemma4prime}
as well as its complete proof,
and provide some irrationality results for 
certain infinite series involving the classical arithmetic functions.  
Before stating our results, we need some preparation. 

\begin{definition}\label{def:PisotSalem} 
A \emph{Pisot number} is
a real algebraic integer greater than $1$ 
whose Galois conjugates over $\mathbb{Q}$
other than itself have absolute value less than $1$.
A \emph{Salem number} is
a real algebraic integer greater than $1$
whose Galois conjugates over $\mathbb{Q}$
other than itself have absolute value no greater than $1$,
and at least one of which has absolute value exactly $1$. 
\end{definition}

Note that every rational integer greater than $1$ is a Pisot number.  
For an algebraic number $\alpha$, 
let $\house{\alpha}$ denote the maximal absolute value
of the Galois conjugates of $\alpha$ over $\mathbb{Q}$.
In particular, we have $\house{\alpha}=|\alpha|$ for any rational number $\alpha$,
and so the function $S_{\bm{a}}(x)$ defined in \cref{def:SNx}
can be extended to a sequence $\bm{a}\coloneqq(a(n))_{n=1}^{\infty}$ of algebraic numbers as 
\begin{equation}
\label{def:SNx_extended}
S_{\bm{a}}(x)
\coloneqq
\sum_{n<x}\house{a(n)}. 
\end{equation}

Moreover, for a sequence $\bm{c}\coloneqq (c(n))_{n=1}^{\infty}$ of complex numbers
and for real numbers $q,x>1$ and $z\geq0$, we define the sets
$\mathcal{N}_{\bm{c}}$ and $\mathcal{N}_{\bm{c}}(x)$ by \cref{def:N} and \cref{def:SNx},
respectively, and  
\begin{equation}
\label{def:R}
R_{\bm{c}}(q,x,z)
\coloneqq
\sum_{n<x}
\sum_{j\ge z}
\frac{|c(n+j)|}{q^{j}},
\end{equation}
provided that the sum converges.
Roughly speaking,
$R_{\bm{c}}(q,x,z)$ measures the decay of the ``tail'' series on average.

Our fundamental irrationality criterion is the following.
\begin{theorem}
\label{thm:criterion_irrationality_original}
Let $q$ be a Pisot or Salem number of degree $d$ over $\mathbb{Q}$. 
Let  
\[
\bm{a}\coloneqq(a(n))_{n=1}^{\infty}
\quad\text{and}\quad
\bm{b}\coloneqq(b(n))_{n=1}^{\infty}
\]
be sequences of algebraic integers of the field $\mathbb{Q}(q)$
such that $a(n)\ge0$ for $n\ge 1$ and $\#\mathcal{N}_{\bm{a}}=\infty$. 
Suppose that there exist sequences
$(x_n)_{n=1}^{\infty}$, $(y_n)_{n=1}^{\infty}$, $(z_n)_{n=1}^{\infty}$
of real numbers $\ge1$ and a constant $\eta\in(0,1]$
satisfying the following conditions as $n\to\infty$\textup{:}
\newcounter{thm:criterion_irrationality_original:i}
\setcounter{thm:criterion_irrationality_original:i}{0}
\begin{enumerate}[leftmargin=*]
\setcounter{enumi}{\value{thm:criterion_irrationality_original:i}}
\item\label{thm:criterion_irrationality_original:x_infty}
$x_{n}\to\infty$. 
\item\label{thm:criterion_irrationality_original:S_house_bound}
$S_{\bm{a}}(x_{n}), S_{\bm{b}}(x_{n})=O(y_{n})$. 
\item\label{thm:criterion_irrationality_original:sparsity}
$\#\mathcal{N}_{\bm{a}}(x_n), \#\mathcal{N}_{\bm{b}}(x_n)
=o\left({x_n}/{z_n}\right)$.
\item\label{thm:criterion_irrationality_original:average_decay}
$R_{\bm{a}}(q,\eta x_{n},z_{n}), R_{\bm{b}}(q,\eta x_{n},z_{n})
=
o(x_{n}/y_{n}^{d-1})$.
\setcounter{thm:criterion_irrationality_original:i}{\value{enumi}}
\end{enumerate}
Moreover, if $\#\mathcal{N}_{\bm{b}}=\infty$, we assume that 
there exist constants $\Delta,L>1$ satisfying 
\begin{enumerate}[leftmargin=*]
\setcounter{enumi}{\value{thm:criterion_irrationality_original:i}}
\item\label{thm:criterion_irrationality_original:interlace}
For any consecutive indices $m<m_{+}$ in $\mathcal{N}_{\bm{b}}$
and any real number $\mu\ge L$ with $m+\Delta\mu<m_{+}$,
we have $\mathcal{N}_{\bm{a}}\cap [m+\mu,m+\Delta\mu)\neq\emptyset$.
\setcounter{thm:criterion_irrationality_original:i}{\value{enumi}}
\end{enumerate}
Then we have
\begin{equation}\label{thm:criterion_irrationality_original:series}
\sum_{n=1}^{\infty}
\frac{a(n)+b(n)}{q^{n}}
\not\in\mathbb{Q}(q).
\end{equation}
\end{theorem}

Condition~\cref{thm:criterion_irrationality_original:average_decay} 
in \cref{thm:criterion_irrationality_original}  
guarantees the convergence of the series \cref{thm:criterion_irrationality_original:series}.
Since this condition is not so convenient for application,
we replace condition \cref{thm:criterion_irrationality_original:average_decay}
with more concise conditions to obtain the following irrationality criterion.

\begin{theorem}\label{thm:criterion_irrationality_prepared}
Let $q$ be a Pisot or Salem number of degree $d$ over $\mathbb{Q}$ 
and $\bm{a}$, $\bm{b}$ be sequences of algebraic integers of $\mathbb{Q}(q)$
such that $a(n)\ge0$ for $n\ge 1$, $\#\mathcal{N}_{\bm{a}}=\infty$ and
\newcounter{thm:criterion_irrationality_prepared:i}
\setcounter{thm:criterion_irrationality_prepared:i}{0}
\begin{enumerate}[label=\textup{(iv-\arabic*)}, leftmargin=*]
\setcounter{enumi}{\value{thm:criterion_irrationality_prepared:i}}
\item
\label{thm:criterion_irrationality_prepared:pointwise}
$\rho\coloneqq 
\varlimsup_{n\to\infty}\max\{\house{a(n)}, \house{b(n)}\}^{1/n}<q$.
\setcounter{thm:criterion_irrationality_prepared:i}{\value{enumi}}
\end{enumerate}
Suppose that there exist sequences
$(x_n)_{n=1}^{\infty}$, $(y_n)_{n=1}^{\infty}$, $(z_{n})_{n=1}^{\infty}$
of real numbers $\ge1$ satisfying conditions
\cref{thm:criterion_irrationality_original:x_infty}, 
\cref{thm:criterion_irrationality_original:S_house_bound},  
\cref{thm:criterion_irrationality_original:sparsity} 
in \cref{thm:criterion_irrationality_original} and
\begin{enumerate}[label=\textup{(iv-\arabic*)}, leftmargin=*]
\setcounter{enumi}{\value{thm:criterion_irrationality_prepared:i}}
\item
\label{thm:criterion_irrationality_prepared:yn_bound}
$\varlimsup_{n\to\infty}y_{n}^{{(d-1)}/{x_{n}}}<q/\rho$.
\item
\label{thm:criterion_irrationality_prepared:moderate_growth}
$\sum_{m<x_n}a(m), \sum_{m<x_n}|b(m)|
=
o(q^{z_{n}}x_{n}/y_{n}^{d-1})$
as $n\to\infty$.
\setcounter{thm:criterion_irrationality_prepared:i}{\value{enumi}}
\end{enumerate}
Moreover, if $\#\mathcal{N}_{\bm{b}}=\infty$,
we assume condition~\cref{thm:criterion_irrationality_original:interlace}
in \cref{thm:criterion_irrationality_original}. 
Then we obtain the conclusion~\cref{thm:criterion_irrationality_original:series}
of \cref{thm:criterion_irrationality_original}. 
\end{theorem}

As applications of \cref{thm:criterion_irrationality_prepared},
we obtain the following Corollaries~\ref{cor:MinkowskiSum} and \ref{cor:degree_ell_irrationality}. 
For non-empty sets $\mathcal{A}$ and $\mathcal{B}$ of non-negative integers,
we define the sumset $\mathcal{A}+\mathcal{B}$ of $\mathcal{A}$ and $\mathcal{B}$ by 
\[
\mathcal{A}+\mathcal{B}\coloneqq 
\{m+n\mid m\in\mathcal{A},\ n\in\mathcal{B}\}.
\]
\begin{corollary}
\label{cor:MinkowskiSum}
Let $q$ be a Pisot or Salem number. 
Let $\bm{u}$ and $\bm{v}$
be sequences of algebraic integers of $\mathbb{Q}(q)$ with $u(n)>0$ for $n\ge 1$. 
Let $\mathcal{A}$ and $\mathcal{B}$
be non-empty sets of non-negative integers with $\# \mathcal{A}=\infty$. 
Suppose that
\begin{equation}\label{cor:MinkowskiSum:sparsity}
\#\bigl((\mathcal{A}+\mathcal{B})\cap [1,n)\bigr)
\cdot
\log\max\{S_{\bm{u}}(n),S_{\bm{v}}(n)\}
=
o(n)
\quad\text{as $n\to\infty$}, 
\end{equation}
and moreover, if $\# \mathcal{B}=\infty$, we assume that there exist constants $\Delta,L>1$ satisfying 
\begin{align}\label{cor:MinkowskiSum:interlace}
\mathcal{A}\cap[\mu,\Delta \mu)\ne\emptyset
\quad\text{for any $\mu\ge L$}. 
\end{align}
Then we have 
\[
\sum_{n\in\mathcal{A}+\mathcal{B}}^{}\frac{u(n)}{q^n}+
\sum_{n\in\mathcal{B}}^{}\frac{v(n)}{q^n}\not\in\mathbb{Q}(q).
\]
\end{corollary}

\begin{example}\label{example:sum_of_two_powers}
Let $q$ be a Pisot or Salem number
and $\alpha,\beta>1$ be real numbers with $\alpha^{-1}+\beta^{-1}<1$.
Let $\bm{u}$ and $\bm{v}$ be sequences of algebraic integers of $\mathbb{Q}(q)$
such that $u(n)>0$ for $n\ge 1$ and 
\[
\log\max\{S_{\bm{u}}(n),S_{\bm{v}}(n)\}
=
o(n^{1-(\alpha^{-1}+\beta^{-1})})
\quad\text{as $n\to\infty$}
\]
and let
\[
\mathcal{A}
\coloneqq
\{\lfloor n^{\alpha}\rfloor\mid n\in\mathbb{Z}_{\ge0}\} 
\quad\text{and}\quad 
\mathcal{B}
\coloneqq
\{\lfloor n^{\beta}\rfloor\mid n\in\mathbb{Z}_{\ge0}\}.
\]
Then, we have 
\[
\sum_{n\in\mathcal{A}+\mathcal{B}}
\frac{u(n)}{q^n}
+
\sum_{n\in\mathcal{B}}
\frac{v(n)}{q^n}
\not\in\mathbb{Q}(q).
\]
For example, putting $q=1+\sqrt{2}$ and $\alpha=\beta=3$, we obtain  
\[
\sum_{n:\text{sum of two cubes}}
\frac{1}{(1+\sqrt{2})^n}
+
\sum_{n=0}^{\infty}
\frac{1}{(1+\sqrt{2})^{n^3}}
\notin\mathbb{Q}(\sqrt{2}). 
\]
\end{example}

\begin{corollary}\label{cor:degree_ell_irrationality}
Let $q>1$ be a Pisot or Salem number
and $(a(n))_{n=1}^{\infty}$ be a sequence of algebraic integers of $\mathbb{Q}(q)$
such that $a(n)\geq0$ for $n\ge1$ and $\#\mathcal{N}_{\bm{a}}=\infty$. 
Write the elements of $\mathcal{N}_{\bm{a}}$
as $n_1,n_2,\dots$ so that $(n_{k})_{k=1}^{\infty}$ is strictly increasing. 
Suppose that
\begin{equation}
\label{cor:degree_ell_irrationality:coefficient_bound}
\varlimsup_{k\to\infty}\house{a(n_{k})}^{\frac{1}{n_{k}}}<q
\end{equation}
and there exists a positive integer $\ell$ such that 
\begin{equation}\label{cor:degree_ell_irrationality:cond}
\varlimsup_{k\to \infty}
\frac{n_k}{k^\ell\log Q(n_{k})}
=
\infty
\end{equation}
with $Q(x)\coloneqq\max(x,\max_{m\le x}\house{a(m)})$.
Then, the number $\sum_{n=1}^{\infty}a(n)/{q^{n}}$
is transcendental or algebraic over $\mathbb{Q}(q)$ of degree at least $\ell+1$.
\end{corollary}

By taking the contraposition of \cref{cor:degree_ell_irrationality},
we can deduce the following:\ 
under the setting of \cref{cor:degree_ell_irrationality}
with a stronger assumption $\log\house{a(n_{k})}=O(\log n_{k})$ as $k\to\infty$
instead of \cref{cor:degree_ell_irrationality:coefficient_bound},
if $\sum_{n=1}^{\infty}a(n)/{q^{n}}$ is algebraic of degree $\ell$ over $\mathbb{Q}(q)$,
then we have
\begin{equation}
\label{cor:degree_ell_irrationality:non_zero_digits}
\#\mathcal{N}_{\bm{a}}(n)
\gg
\frac{n^{1/\ell}}{(\log n)^{1/\ell}}
\quad\text{as $n\to\infty$}
\end{equation}
which is equivalent to the negation of \cref{cor:degree_ell_irrationality:cond}.
This estimate of the number of ``non-zero digits''
is a partial improvement of Theorem~2.2 of \cite{Kaneko:non_zero_digits_Padova},
where $a(n)$ is originally assumed to be bounded and taking values in the rational integers.
However, note also that Theorem 2.2 of \cite{Kaneko:non_zero_digits_Padova}
gives effective lower bounds for the implicit constant
in \cref{cor:degree_ell_irrationality:non_zero_digits}.

\begin{example}\label{example:polynomial_over_power}
Let $q>1$ be a Pisot or Salem number and 
$f(X)$ be a polynomial with coefficients in $\mathbb{Q}(q)$
and having positive leading coefficient.
Then, for any positive integer $\ell\geq2$, the numbers
\[
\sum_{n=1}^{\infty}
\frac{f(n)}{q^{n^{\ell+1}}}
\quad\text{and}\quad 
\sum_{p\text{:prime}}^{\infty}
\frac{f(p)}{q^{p^{\ell}}}
\]
are transcendental or algebraic over $\mathbb{Q}(q)$ of degree at least $\ell+1$. 
\end{example}

If $q\geq2$ is a rational integer,
then the conditions in Theorems~\ref{thm:criterion_irrationality_original} and 
\ref{thm:criterion_irrationality_prepared} become quite concise. 
Indeed, we can ignore all conditions and factors related to $y_{n}$ 
since every rational integer has degree $d=1$. 
In this case, \cref{thm:criterion_irrationality_prepared} takes the following form.

\begin{theorem}\label{thm:criterion_irrationality_rational}
Let $t\ge2$ be a rational integer
and $\bm{a}$, $\bm{b}$ be sequences of rational integers 
such that $a(n)\ge0$ for $n\geq1$, $\#\mathcal{N}_{\bm{a}}=\infty$ and 
\newcounter{thm:criterion_irrationality_rational:i}
\setcounter{thm:criterion_irrationality_rational:i}{0}
\begin{enumerate}[leftmargin=*]
\setcounter{enumi}{\value{thm:criterion_irrationality_rational:i}}
\item\label{thm:criterion_irrationality_rational:pointwise}
$\varlimsup_{n\to\infty}\max(a(n),|b(n)|)^{1/n}<t$.
\setcounter{thm:criterion_irrationality_rational:i}{\value{enumi}}
\end{enumerate}
Suppose that there exist sequences
$(x_n)_{n=1}^{\infty}$ and $(z_{n})_{n=1}^{\infty}$ of real numbers $\ge1$
satisfying the following conditions as $n\to\infty$\textup{:}
\begin{enumerate}[leftmargin=*]
\setcounter{enumi}{\value{thm:criterion_irrationality_rational:i}}
\item\label{thm:criterion_irrationality_rational:x_infty}
$x_{n}\to\infty$. 
\item\label{thm:criterion_irrationality_rational:moderate_growth}
$S_{\bm{a}}(x_{n}), S_{\bm{b}}(x_{n})
=
o(t^{z_{n}}x_{n})$. 
\item\label{thm:criterion_irrationality_rational:sparsity}
$\#\mathcal{N}_{\bm{a}}(x_{n}),\#\mathcal{N}_{\bm{b}}(x_{n})
=
o({x_{n}}/{z_{n}})$.
\setcounter{thm:criterion_irrationality_rational:i}{\value{enumi}}
\end{enumerate}
Moreover, if $\#\mathcal{N}_{\bm{b}}=\infty$, we assume 
condition~\cref{thm:criterion_irrationality_original:interlace}  
in \cref{thm:criterion_irrationality_original}. 
Then, $\sum_{n=1}^{\infty}{(a(n)+b(n))}/{t^{n}}$ is irrational. 
\end{theorem}

Erd\H{o}s's irrationality criterion \cref{prethm:ErdosLemma4prime}
will be derived from \cref{thm:criterion_irrationality_rational}
in \cref{sec:proof_corollary}.

As applications of \cref{thm:criterion_irrationality_rational},
we obtain the following irrationality results. 
Let $\sigma(n)$ and $\phi(n)$ be the sum of divisor and Euler's totient functions
defined in \cref{def:sigma_phi}, respectively. 

\begin{corollary}
\label{cor:chimera}
Let $\bm{f}=(f(n))_{n=1}^{\infty}$ be a sequence of non-negative rational integers 
with $\#\mathcal{N}_{\bm{f}}=\infty$. 
Suppose that there exists a positive constant $\delta$ such that 
\begin{equation}
\label{cor:chimera:cond}
\sum_{n\le x}f(n)
=
O(x(\log x)^{\delta})
\quad\text{as $x\to\infty$}.
\end{equation}
Then, for any rational integer $t\geq2$, the numbers 
\begin{equation}
\label{cor:chimera:series}
\sum_{n=1}^{\infty}
\frac{f(n)}{t^{\sigma(n)}}
\quad\text{and}\quad
\sum_{n=1}^{\infty}
\frac{f(n)}{t^{\phi(n)}}
\end{equation}
are both irrational. 
\end{corollary}

\begin{example}\label{example:chimera}
Let $d(n)$ for $n\ge 1$ denote the number of divisors of $n$.  
Then, for a fixed non-negative integer $k$,
we have $\sum_{n\le x}d(n)^{k}=O(x(\log x)^{2^k-1})$ as $x\to\infty$ 
(cf. \cite[p.~61]{MV:text}).  
Hence, for any rational integers $t\ge2$ and $k\geq0$, the numbers 
\[
\sum_{n=1}^{\infty}
\frac{d(n)^k}{t^{\sigma(n)}}
\quad\text{and}\quad
\sum_{n=1}^{\infty}
\frac{d(n)^k}{t^{\phi(n)}}
\]
are both irrational. 
\end{example}

Note that \cref{cor:chimera} with $k=0$ gives the irrationality of the numbers \cref{sigma_phi}. 
\begin{example}\label{example:minus_omega_Omega}
For a positive integer $n$,
let $\omega(n)$ be the number of distinct prime factors of $n$
and $\Omega(n)$ be the number of prime factors of $n$ counted with multiplicities. 
Then the two sequences
$f(n)\coloneqq2^{\omega(n)},2^{\Omega(n)}$
for $n\ge 1$ satisfy \cref{cor:chimera:cond} with $\delta=2$ 
(see \cite[p.~42]{MV:text} for precise asymptotic formulae).
Hence, the numbers 
\[
\sum_{n=1}^{\infty}
\frac{1}{2^{\sigma(n)-\omega(n)}},\quad
\sum_{n=1}^{\infty}
\frac{1}{2^{\sigma(n)-\Omega(n)}},\quad
\sum_{n=1}^{\infty}
\frac{1}{2^{\phi(n)-\omega(n)}},\quad
\sum_{n=1}^{\infty}
\frac{1}{2^{\phi(n)-\Omega(n)}}
\]
are all irrational. 
\end{example}

The present paper is organized as follows. 
In \cref{sec:proof_irrationality},
we first prepare some lemmas, and then prove \cref{thm:criterion_irrationality_original}. 
The proof of \cref{thm:criterion_irrationality_original} is based on 
a refinement of Erd\H{o}s's argument in \cite{Erdos:PowerSeriesIrrationalitySigmaPhi}. 
In \cref{subsec:deduction_prepared}, we derive \cref{thm:criterion_irrationality_prepared}
as a consequence of \cref{thm:criterion_irrationality_original}. 
\cref{sec:addcor} is devoted to the proofs of 
\cref{cor:MinkowskiSum,cor:degree_ell_irrationality}. 
In \cref{sec:proof_corollary}, we prove Erd\H{o}s's \cref{prethm:ErdosLemma4prime} and \cref{cor:chimera}. 
Throughout this paper,
let $q>1$ be a fixed Pisot or Salem number of degree $d$ over $\mathbb{Q}$.

\section{Proof of \cref{thm:criterion_irrationality_original}}
\label{sec:proof_irrationality}

For a sequence 
$\bm{c}=(c(n))_{n=1}^{\infty}$ of complex numbers
and for an integer $N\ge1$, we define 
\begin{equation}\label{def:xi_N}
\xi(\bm{c})
\coloneqq 
\sum_{n=1}^{\infty}\frac{c(n)}{q^{n}},
\end{equation}
provided that $\xi(\bm{c})$ is convergent, and 
\begin{equation}
\label{def:xiN}
\xi_{N}(\bm{c})
\coloneqq 
q^{N}
\biggl(
\xi(\bm{c})-\sum_{n<N}\frac{c(n)}{q^{n}}
\biggr)
=
q^{N}\sum_{n\ge N}\frac{c(n)}{q^{n}}
=
\sum_{j\ge 0}\frac{c(N+j)}{q^{j}}.
\end{equation}
For brevity, write $|\bm{c}|\coloneqq(|c(n)|)_{n=1}^{\infty}$. 
The convergence of $\xi(|\bm{c}|)$
is equivalent to the absolute convergence of $\xi(\bm{c})$. 
We first prove the following \cref{lem:fundamental_criterion}
which plays an essential role in the proof of \cref{thm:criterion_irrationality_original}. 

\begin{lemma}
\label{lem:fundamental_criterion}
Let $\bm{a}$ and $\bm{b}$ be sequences of algebraic integers of $\mathbb{Q}(q)$
such that $a(n)\ge0$ for $n\ge1$, $\#\mathcal{N}_{\bm{a}}=\infty$ and 
the series $\xi(\bm{a}), \xi(|\bm{b}|)$ are convergent. 
Suppose that there exist infinitely many integers $N\ge1$
satisfying $\xi_{N}(\bm{a})>\xi_{N}(|\bm{b}|)$, 
and moreover, we assume
\begin{equation}
\label{lem:fundamental_criterion:decaying_tail}
\varliminf_{\substack{
N\to\infty\\
\xi_{N}(\bm{a})>\xi_{N}(|\bm{b}|)
}}
\xi_{N}(|\bm{a}+\bm{b}|)\cdot S_{\bm{a}+\bm{b}}(N)^{d-1}
=
0,
\end{equation}
where $\bm{a}+\bm{b}\coloneqq(a(n)+b(n))_{n=1}^{\infty}$. 
Then we have $\xi(\bm{a}+\bm{b})\not\in \mathbb{Q}(q)$.
\end{lemma}
\begin{proof} 
Let $\mathcal{O}_{K}$ be the ring of integers in $K\coloneqq\mathbb{Q}(q)$. 
Assume to the contrary that $\xi(\bm{a}+\bm{b})\in\mathbb{Q}(q)$
and put $\xi(\bm{a}+\bm{b})=\gamma/u$,
where $\gamma\in \mathcal{O}_{K}$ and $u\in\mathbb{Z}_{\ge1}$. 
Let $N$ be an integer satisfying $\xi_{N}(\bm{a})>\xi_{N}(|\bm{b}|)$.  
Then we have 
\[
|\xi_{N}(\bm{a}+\bm{b})|
\ge 
\xi_{N}(\bm{a})-\xi_{N}(|\bm{b}|)
>
0
\]
and
\begin{equation}\label{lem:fundamental_criterion:ux_N}
u\xi_{N}(\bm{a}+\bm{b})
=
q^{N}
\gamma
-
u\sum_{n<N}(a(n)+b(n))q^{N-n}
\in\mathcal{O}_{K}.
\end{equation}
Hence, by taking the norm with respect to the extension $K/\mathbb{Q}$, we obtain 
\begin{equation}\label{lem:fundamental_criterion:norm}
|N_{K/\mathbb{Q}}(u \xi_{N}(\bm{a}+\bm{b}))|\ge1.
\end{equation}
On the other hand, by \cref{def:xiN}, we have
\begin{equation}\label{lem:fundamental_criterion:error}
|u\xi_{N}(\bm{a}+\bm{b})|
\le
uq^{N}
\sum_{n\ge N}
\frac{|a(n)+b(n)|}{q^{n}}
=
u\xi_{N}(|\bm{a}+\bm{b}|).
\end{equation}
Moreover, since $\xi_{N}(\bm{a})>\xi_{N}(|\bm{b}|)$,
there exists an integer $n\ge N$ such that $a(n)+b(n)\neq0$, 
and hence, $S_{\bm{a}+\bm{b}}(N)\geq1$ for every large integer $N$. 
Thus, for any embedding $\sigma:\mathbb{Q}(q)\to \mathbb{C}$ with $\sigma(q)\neq q$,
by \cref{lem:fundamental_criterion:ux_N}, we obtain
\begin{align}
|(u \xi_{N}(\bm{a}+\bm{b}))^\sigma|
&=
\biggl|(q^\sigma)^{N}\gamma^\sigma
-u\sum_{n<N}
(a(n)+b(n))^\sigma(q^\sigma)^{N-n}
\biggr|\\
&\le
\house{\gamma}
+
u\sum_{n<N}\house{a(n)+b(n)}\\
&=
O(S_{\bm{a}+\bm{b}}(N))
\quad\text{as $N\to \infty$}\label{lem:fundamental_criterion:error_conjugate},
\end{align}
where we used $|q^\sigma|\leq 1$ since $q$ is a Pisot or Salem number. 
Therefore, by
\cref{lem:fundamental_criterion:norm},
\cref{lem:fundamental_criterion:error} and
\cref{lem:fundamental_criterion:error_conjugate},
\[
1
\le
|N_{K/\mathbb{Q}}(u\xi_{N}(\bm{a}+\bm{b}))|
=
\prod_{\sigma}
|(u\xi_{N}(\bm{a}+\bm{b}))^{\sigma}|
=
O(
\xi_{N}(|\bm{a}+\bm{b}|)
\cdot
S_{\bm{a}+\bm{b}}(N)^{d-1}
)
\quad\text{as $N\to\infty$},
\]
which contradicts \cref{lem:fundamental_criterion:decaying_tail}. 
The proof of \cref{lem:fundamental_criterion} is completed. 
\end{proof}

In what follows, for brevity, we write \cref{def:R} as 
\[
R_{\bm{a}}(x,z)\coloneqq 
R_{\bm{a}}(q,x,z)
=
\sum_{n<x}
\sum_{j\ge z}
\frac{|a(n+j)|}{q^{j}}.
\]

\begin{lemma}
\label{lem:small_tail}
Let $\bm{c}=(c(n))_{n=1}^{\infty}$ be a sequence of complex numbers. 
Suppose that there exist sequences
$(x_n)_{n=1}^{\infty}$, $(w_{n})_{n=1}^{\infty}$, $(z_{n})_{n=1}^{\infty}$
of real numbers $\ge1$ and a constant $\eta\in(0,1]$
satisfying the following conditions as $n\to\infty$\textup{:}
\begin{enumerate}[leftmargin=*]
\item\label{lem:small_tail:x_infty}
$x_{n}\to\infty$.
\item\label{lem:small_tail:sparsity}
$\#\mathcal{N}_{\bm{c}}(x_{n})=o\left({x_{n}}/{z_{n}}\right)$.
\item\label{lem:small_tail:average_decay}
$R_{\bm{c}}(\eta x_{n},z_{n})
=
o(x_{n}/w_{n})$.
\end{enumerate}
Then, for any fixed constant $\delta>0$, we have
\[
\#\{
N\in[1,\eta x_{n})\cap\mathbb{Z}
\mid
w_n \xi_{N}(|\bm{c}|)<\delta
\}
=
\eta x_n+o(x_{n})
\quad\text{as $n\to\infty$}.
\]
\end{lemma}
\begin{proof} 
We may assume $\#\mathcal{N}_{\bm{c}}(x_n)\ge1$ for large $n$
since otherwise $\xi_{N}(|\bm{c}|)=0$ for $N\ge1$ by condition~\cref{lem:small_tail:x_infty}
and the assertion is clear.  
Let $n$ be an integer sufficiently large and
\[
\mathscr{I}
\coloneqq
[1,\eta x_{n})\cap\mathbb{Z}.
\]  
Let $\delta>0$ be a fixed constant and       
\begin{align}
\mathscr{I}_{\mathrm{good}}
&\coloneqq
\{N\in\mathscr{I}
\mid
w_n \xi_{N}(|\bm{c}|)<\delta
\},\\
\mathscr{I}_{1}
&\coloneqq
\{N\in\mathscr{I}
\mid
c(N+j)=0\ \text{for all integers $j\in[0,z_{n})$}
\},\\
\mathscr{I}_{2}
&\coloneqq
\{N\in\mathscr{I}
\mid
c(N+j)\neq0\ \text{for some integer $j\in[0,z_{n})$}
\}. 
\end{align}
Then we have    
\begin{equation}\label{lem:small_tail:decomp}
\#\mathscr{I}_{\mathrm{good}}
=
\#\mathscr{I}
-
\#(\mathscr{I}_{1}\setminus\mathscr{I}_{\mathrm{good}})
-
\#(\mathscr{I}_{2}\setminus\mathscr{I}_{\mathrm{good}})
\end{equation}
since $\mathscr{I}=\mathscr{I}_{1}\sqcup \mathscr{I}_{2}$. 
Since $N\in\mathscr{I}\setminus\mathscr{I}_{\mathrm{good}}$
implies $w_{n}\xi_{N}(|\bm{c}|)\ge\delta$,
by recalling \cref{def:xiN}, we have
\begin{align}
\#(\mathscr{I}_{1}\setminus \mathscr{I}_{\mathrm{good}})
&=
\sum_{N\in\mathscr{I}_{1}\setminus \mathscr{I}_{\mathrm{good}}}
1\\
&\le
\sum_{N\in\mathscr{I}_{1}\setminus \mathscr{I}_{\mathrm{good}}}
\delta^{-1}w_n\xi_{N}(|\bm{c}|)
=
\delta^{-1}w_{n}
\sum_{N\in\mathscr{I}_{1}\setminus \mathscr{I}_{\mathrm{good}}}
\sum_{j\ge 0}\frac{|c(N+j)|}{q^{j}}.
\end{align}
Since $N\in\mathscr{I}_{1}$ implies $c(N+j)=0$ for $j\in[0,z_{n})$,
by condition~\cref{lem:small_tail:average_decay} and $\mathscr{I}_{1}\subset[1,\eta x_{n})$, we obtain 
\begin{align}
\#(\mathscr{I}_{1}\setminus\mathscr{I}_{\mathrm{good}})
&\le
\delta^{-1}w_n\sum_{N\in\mathscr{I}_{1}\setminus\mathscr{I}_{\mathrm{good}}}
\sum_{j\ge z_{n}}
\frac{|c(N+j)|}{q^{j}}\\
&\le
\delta^{-1}w_n\sum_{N<\eta x_{n}}
\sum_{j\ge z_{n}}
\frac{|c(N+j)|}{q^{j}}\\
&=
\delta^{-1}w_n\cdot R_{\bm{c}}(\eta x_{n},z_{n})
=
o(x_n)
\quad\text{as $n\to\infty$}.
\label{lem:small_tail:T_H}
\end{align}
Moreover, by recalling the definition of $\mathscr{I}_{2}$, we have
\begin{align}
\#(\mathscr{I}_{2}\cap[0,\eta x_{n}-z_{n}))
&=
\sum_{\substack{
N<\eta x_{n}-z_{n}\\
c(N+j)\neq0\ (\exists j\in[0,z_{n})\cap\mathbb{Z})
}}
1\\
&\le
\sum_{N<\eta x_{n}-z_{n}}
\sum_{\substack{
0\le j<z_{n}\\
c(N+j)\neq0
}}
1
=
\sum_{0\le j<z_{n}}
\sum_{\substack{
N<\eta x_{n}-z_{n}\\
c(N+j)\neq0
}}
1\\
&\le
(z_{n}+1)\#\mathcal{N}_{\bm{c}}(\eta x_{n})
\le
2z_{n}\#\mathcal{N}_{\bm{c}}(x_{n}).
\end{align}
Thus,
by conditions~\cref{lem:small_tail:x_infty} and \cref{lem:small_tail:sparsity},
we have
\begin{align}
\#(\mathscr{I}_{2}\setminus\mathscr{I}_{\mathrm{good}})
\le 
\#\mathscr{I}_{2}
&=
\#(\mathscr{I}_{2}\cap[0,\eta x_{n}-z_{n}))
+
\#(\mathscr{I}_{2}\cap[\eta x_{n}-z_{n},\eta x_{n}))\\
&\le
2z_{n}\#\mathcal{N}_{\bm{c}}(x_{n})+z_{n}+1
\le
4z_{n}\#\mathcal{N}_{\bm{c}}(x_{n})=o(x_n)\quad\text{as $n\to\infty$}, \label{lem:small_tail:E}
\end{align}
where at the last inequality, we used $\#\mathcal{N}_{\bm{c}}(x_n)\ge1$ for every large integer $n$.
Hence \cref{lem:small_tail} follows from \cref{lem:small_tail:decomp} together with 
\cref{lem:small_tail:T_H}, \cref{lem:small_tail:E} and $\#\mathscr{I}=\eta x_n+O(1)$ as $n\to\infty$. 
\end{proof}

\begin{lemma}
\label{lem:perturbation_removal}
Let $\bm{a}$ and $\bm{b}$ be sequences of algebraic integers of $\mathbb{Q}(q)$
such that $\#\mathcal{N}_{\bm{a}}=\infty$. 
Suppose that there exist sequences
$(x_n)_{n=1}^{\infty}$, $(y_n)_{n=1}^{\infty}$, $(z_{n})_{n=1}^{\infty}$
of real numbers $\ge1$ and a constant $\eta\in(0,1]$
satisfying the following conditions as $n\to\infty$\textup{:}
\newcounter{lem:perturbation_removal:i}
\setcounter{lem:perturbation_removal:i}{0}
\begin{enumerate}[leftmargin=*]
\setcounter{enumi}{\value{lem:perturbation_removal:i}}
\item\label{lem:perturbation_removal:x_infty}
$x_{n}\to\infty$.
\item\label{lem:perturbation_removal:S_house_bound}
$S_{\bm{a}}(x_{n})=O(y_{n})$.
\item\label{lem:perturbation_removal:sparsity}
$\#\mathcal{N}_{\bm{b}}(x_{n})=o\left({x_{n}}/{z_{n}}\right)$.
\item\label{lem:perturbation_removal:average_decay}
$R_{\bm{b}}(\eta x_{n},z_{n})
=
o(x_{n}/y_{n}^{d-1})$.
\setcounter{lem:perturbation_removal:i}{\value{enumi}}
\end{enumerate}
Moreover, when $\#\mathcal{N}_{\bm{b}}=\infty$, 
we assume that 
there exist constants $\Delta,L>1$ satisfying the condition
\begin{enumerate}
\setcounter{enumi}{\value{lem:perturbation_removal:i}}
\item\label{lem:perturbation_removal:interlace}
For any consecutive indices $m<m_{+}$ in $\mathcal{N}_{\bm{b}}$
and any real number $\mu\ge L$ with $m+\Delta\mu<m_{+}$,
we have $\mathcal{N}_{\bm{a}}\cap [m+\mu,m+\Delta\mu)\neq\emptyset$.
\end{enumerate}
Then there exists a constant $C>0$ such that 
\begin{equation}\label{lem:perturbation_removal:lower_bound}
\#\{
N\in[1,\eta x_{n})\cap\mathbb{Z}
\mid
\xi_{N}(|\bm{a}|)>\xi_{N}(|\bm{b}|)
\}
\ge
Cx_{n}
\end{equation}
holds for every large integer $n$. 
\end{lemma}
\begin{proof}
If $\#\mathcal{N}_{\bm{b}}<\infty$, 
then $\xi_{N}(|\bm{b}|)=0$ for every large integer $N$. 
On the other hand,
$\xi_{N}(|\bm{a}|)>0$ for every integer $N\ge1$
as $\#\mathcal{N}_{\bm{a}}=\infty$. 
Hence, the inequality \cref{lem:perturbation_removal:lower_bound}
is valid for any constant $C$ with $0<C<\eta$.
Thus, in what follows, we assume $\#\mathcal{N}_{\bm{b}}=\infty$.

Let $n\ge1$ be an integer sufficiently large and   
\[
\mathscr{J}
\coloneqq
\mathcal{N}_{\bm{b}}(\eta x_n)
=
\{
m\in[1,\eta x_{n})\cap\mathbb{Z}
\mid
b(m)\neq 0
\},\quad
\overline{\mathscr{J}}
\coloneqq
\mathscr{J}\cup\{\eta x_n\}.
\]
For an integer $m\in\mathscr{J}$,
let $m_{+}$ denote the successor of $m$ in $\overline{\mathscr{J}}$, namely, 
$m_{+}$ is the smallest element in $\overline{\mathscr{J}}$ greater than $m$.
Then we have 
\begin{equation}
\label{lem:perturbation_removal:J_sum}
\sum_{m\in\mathscr{J}}
(m_{+}-m)
=
\eta x_n-\min_{m\in\mathscr{J}}m
=
\eta x_{n}+O(1)
\quad\text{as $n\to\infty$}. 
\end{equation}
By condition~\cref{lem:perturbation_removal:S_house_bound},
there exists a constant $c>1$ such that 
\begin{equation}
\label{lem:perturbation_removal:S_house_bound_explicit}
S_{\bm{a}}(x_{n})
<
cy_{n}
\quad\text{for $n\ge 1$}.
\end{equation}
Let $\Delta$ and $L$ be constants $>1$ given in condition~\cref{lem:perturbation_removal:interlace}. 
Define 
\begin{align}
\mathscr{J}_{\mathrm{bad}}^{(1)}
&\coloneqq
\left\{
m\in\mathscr{J}
\mid
m_{+}-m
\le
2\Delta 
L
\right\},\\
\mathscr{J}_{\mathrm{bad}}^{(2)}
&\coloneqq
\{
m\in\mathscr{J}
\mid y_n^{d-1}\xi_{N}(|\bm{b}|)\ge  c^{1-d}
\ \text{for all integers $N\in\left[m+\tfrac{1}{2}(m_{+}-m),m_{+}\right)$}
\},\\
\mathscr{J}_{\mathrm{good}}
&\coloneqq
\mathscr{J}
\setminus
(
\mathscr{J}_{\mathrm{bad}}^{(1)}
\cup
\mathscr{J}_{\mathrm{bad}}^{(2)}
).
\end{align}
We first show 
\begin{equation}
\label{lem:perturbation_removal:J_good}
\sum_{m\in\mathscr{J}_{\mathrm{good}}}(m_{+}-m)
=
\eta x_{n}+o(x_{n})
\quad\text{as $n\to\infty$}.
\end{equation}
To see this, we estimate $\sum_{m\in\mathscr{J}_{\mathrm{bad}}^{(j)}}(m_{+}-m)$ for $j=1,2$. 
By condition \cref{lem:perturbation_removal:sparsity}, we have 
\begin{equation}\label{lem:perturbation_removal:J_density_zero}
\#\mathscr{J}
\le
\#\mathcal{N}_{\bm{b}}(x_n)=o(x_n)\quad\text{as $n\to\infty$}. 
\end{equation}
Hence, by \cref{lem:perturbation_removal:J_density_zero}
and the definition of $\mathscr{J}_{\mathrm{bad}}^{(1)}$, we have 
\begin{equation}\label{lem:perturbation_removal:J_bad_1}
\sum_{m\in\mathscr{J}_{\mathrm{bad}}^{(1)}}
(m_{+}-m)
\le
2\Delta L\cdot \#\mathscr{J}_{\mathrm{bad}}^{(1)}
\le
2\Delta L\cdot \#\mathscr{J}
=
o(x_{n})
\quad\text{as $n\to\infty$}.
\end{equation}
Moreover, for each integer $m\in\mathscr{J}_{\mathrm{bad}}^{(2)}$, we have 
\[
\sum_{\substack{
m\le N<m_{+}\\
y_n^{d-1}\xi_{N}(|\bm{b}|)\ge c^{1-d}
}}
1
\ge
\sum_{m+\frac{1}{2}(m_{+}-m)\le N<m_{+}}
1
\ge
\frac{1}{2}(m_{+}-m)-1
\]
so that  
\begin{align}
\frac{1}{2}\sum_{m\in\mathscr{J}_{\mathrm{bad}}^{(2)}}(m_{+}-m)
&\le
\sum_{m\in\mathscr{J}_{\mathrm{bad}}^{(2)}}
1
+
\sum_{m\in\mathscr{J}_{\mathrm{bad}}^{(2)}}
\sum_{\substack{
m\le N<m_{+}\\
y_n^{d-1}\xi_{N}(|\bm{b}|)\ge  c^{1-d}
}}
1\\
&\le 
\#\mathscr{J}
+
\#\{
N\in[1,\eta x_{n})\cap\mathbb{Z}
\mid
y_n^{d-1}\xi_{N}(|\bm{b}|)\ge c^{1-d}
\}.\label{lem:perturbation_removal:J_bad_2_prelim}
\end{align}
By conditions 
\cref{lem:perturbation_removal:x_infty}, 
\cref{lem:perturbation_removal:sparsity},
\cref{lem:perturbation_removal:average_decay}, 
the sequences 
$(x_n)_{n=1}^{\infty}$,
$(w_{n})_{n=1}^{\infty}\coloneqq(y_{n}^{d-1})_{n=1}^{\infty}$,
$(z_n)_{n=1}^{\infty}$ and the constant $\eta$ 
satisfy all conditions in \cref{lem:small_tail} for $\bm{c}\coloneqq \bm{b}$. 
Hence, applying \cref{lem:small_tail} with $\delta\coloneqq c^{1-d}$, we obtain  
\begin{equation}\label{lem:perturbation_removal:J_bad_2_second}
\#\{
N\in[1,\eta x_{n})\cap\mathbb{Z}
\mid
y_n^{d-1}\xi_{N}(|\bm{b}|)\ge c^{1-d}
\}
=
o(x_n)
\quad\text{as $n\to\infty$}. 
\end{equation}
Thus, it follows from \cref{lem:perturbation_removal:J_bad_2_prelim}
together with \cref{lem:perturbation_removal:J_density_zero}
and \cref{lem:perturbation_removal:J_bad_2_second} that 
\begin{equation}
\label{lem:perturbation_removal:J_bad_2}
\sum_{m\in\mathscr{J}_{\mathrm{bad}}^{(2)}}(m_{+}-m)
=
o(x_{n})
\quad\text{as $n\to\infty$}.
\end{equation}
Therefore, by \cref{lem:perturbation_removal:J_sum},
\cref{lem:perturbation_removal:J_bad_1}
and \cref{lem:perturbation_removal:J_bad_2}, we obtain 
\[
\sum_{m\in\mathscr{J}_{\mathrm{good}}}
(m_{+}-m)
=
\sum_{m\in\mathscr{J}}
(m_{+}-m)
-
\sum_{m\in\mathscr{J}_{\mathrm{bad}}^{(1)}\cup \mathscr{J}_{\mathrm{bad}}^{(2)}}
(m_{+}-m)
=
\eta x_n+o(x_n)
\quad\text{as $n\to\infty$}. 
\]

Next, for any integer $m\in\mathscr{J}_{\mathrm{good}}$, we prove 
\begin{equation}
\label{lem:perturbation_removal:good}
\xi_{N}(|\bm{a}|)>\xi_{N}(|\bm{b}|)
\quad\text{for all integers $N\in(m,m+\ell(m)]$},
\end{equation}
where 
\begin{equation}
\label{lem:perturbation_removal:ell_def}
\ell(m)
\coloneqq
\frac{1}{2\Delta}(m_{+}-m).
\end{equation}
Fix $m\in\mathscr{J}_{\mathrm{good}}$. 
Since $m\not\in\mathscr{J}_{\mathrm{bad}}^{(1)}$, we have $\ell(m)>L$,
and moreover, by \cref{lem:perturbation_removal:ell_def},
\[
m+\Delta\cdot\ell(m)
=
m+\frac{1}{2}(m_{+}-m)
<
m_{+}.
\]
Hence, by condition~\cref{lem:perturbation_removal:interlace},
there exists a positive integer $u=u(m)\in\mathcal{N}_{\bm{a}}$ such that 
\[
m+\ell(m)\le u<m+\Delta\cdot\ell(m).
\]
Noting that $u<m_{+}\le x_{n}$
and $a(u)$ is a non-zero algebraic integer of $\mathbb{Q}(q)$,
by \cref{lem:perturbation_removal:S_house_bound_explicit}, we have 
\begin{equation}\label{lem:perturbation_removal:alpha_m}
\xi_{u}(|\bm{a}|)
=
\sum_{h\ge0}
\frac{|a(u+h)|}{q^{h}}
\ge
|a(u)|
\ge 
\house{a(u)}^{-(d-1)}
\ge
S_{\bm{a}}(x_{n})^{-(d-1)}
>
(cy_{n})^{-(d-1)}.
\end{equation}
Also, since $m\not\in\mathscr{J}_{\mathrm{bad}}^{(2)}$, 
there exists an integer $k=k(m)\in[m+\frac{1}{2}(m_{+}-m),m_{+})$ satisfying 
\begin{equation}\label{lem:perturbation_removal:small_tail}
\xi_{k}(|\bm{b}|)
\le 
(cy_{n})^{-(d-1)}.
\end{equation}
Moreover,   
\begin{equation}
\label{lem:perturbation_removal:vanishing_b}
b(n)=0\ \text{for all integers $n\in(m,m_{+})$}
\end{equation}
and
\begin{equation}
\label{lem:perturbation_removal:indices_order}
m<m+\ell(m)\le u<m+\Delta \cdot \ell(m)=m+\frac{1}{2}(m_{+}-m)\le k<m_{+}.
\end{equation}
Hence, by \cref{lem:perturbation_removal:small_tail}, \cref{lem:perturbation_removal:vanishing_b} and 
\cref{lem:perturbation_removal:indices_order},   
for any integer $N\in(m,m+\ell(m)]$, we obtain 
\begin{equation}\label{lem:perturbation_removal:beta_m}  
\xi_{N}(|\bm{b}|)
=
q^{N}\sum_{h\geq N}\frac{|b(h)|}{q^h}
=
q^{N}\sum_{h\geq k}\frac{|b(h)|}{q^h}
=
q^{N-k}
\xi_{k}(|\bm{b}|)
\le
q^{N-k}(cy_{n})^{-(d-1)}.
\end{equation}
Thus, by
\cref{lem:perturbation_removal:alpha_m} and
\cref{lem:perturbation_removal:beta_m}, 
for any integer $m\in\mathscr{J}_{\mathrm{good}}$ and
any integer $N\in(m,m+\ell(m)]$, we have 
\[
\xi_{N}(|\bm{a}|)
\ge
q^{N}
\sum_{h\ge u}\frac{|a(h)|}{q^{h}}
=
q^{N-u}
\xi_{u}(|\bm{a}|)
>
q^{N-u}(cy_{n})^{-(d-1)}
\ge 
q^{k-u}
\xi_{N}(|\bm{b}|)
>
\xi_{N}(|\bm{b}|),
\]
and \cref{lem:perturbation_removal:good} is proved. 

Finally, by
\cref{lem:perturbation_removal:J_good},
\cref{lem:perturbation_removal:good} and
\cref{lem:perturbation_removal:ell_def},
we obtain
\begin{align}
\#\{
N\in[1,\eta x_{n})\cap\mathbb{Z}
\mid
\xi_{N}(|\bm{a}|)>\xi_{N}(|\bm{b}|)
\}
&\ge 
\sum_{m\in\mathscr{J}}
\sum_{\substack{
m\le N<m_{+}\\
\xi_{N}(|\bm{a}|)>\xi_{N}(|\bm{b}|)
}}
1\\
&\ge
\sum_{m\in\mathscr{J}_{\mathrm{good}}}
\sum_{\substack{
m<N\le m+\ell(m)\\
\xi_{N}(|\bm{a}|)>\xi_{N}(|\bm{b}|)
}}
1\\
&=
\sum_{m\in\mathscr{J}_{\mathrm{good}}}
\ell(m)
=
\frac{\eta}{2\Delta} x_{n}
+
o(x_{n})
\quad\text{as $n\to\infty$},
\end{align}
which shows that \cref{lem:perturbation_removal:lower_bound} holds
for any constant $C$ with $0<C<\eta/(2\Delta)$.
The proof is completed. 
\end{proof}

Now we prove \cref{thm:criterion_irrationality_original}. 
\begin{proof}[Proof of \cref{thm:criterion_irrationality_original}]
Let $\bm{a}$, $\bm{b}$,
$(x_n)_{n=1}^{\infty}$, $(y_n)_{n=1}^{\infty}$, $(z_n)_{n=1}^{\infty}$, 
$\eta$ be as in \cref{thm:criterion_irrationality_original} and  
$(w_n)_{n=1}^{\infty}\coloneqq(y_n^{d-1})_{n=1}^{\infty}$. 
Then, all conditions of \cref{lem:small_tail} for $\bm{c}\coloneqq \bm{a}+\bm{b}$
and all conditions of \cref{lem:perturbation_removal} for $\bm{a},\bm{b}$ are satisfied. 
Hence, for an arbitrary constant $\epsilon>0$, by \cref{lem:small_tail} we have  
\begin{equation}\label{thm:criterion_irrationality_original:small_tail}
\#\mathcal{F}_{1}
\coloneqq 
\#\{
N\in[1,\eta x_{n})\cap\mathbb{Z}
\mid
y_n^{d-1}\xi_{N}(|\bm{a}+\bm{b}|)<\epsilon
\}
=
\eta x_n+o(x_n)
\quad\text{as $n\to\infty$}.
\end{equation}
Moreover, we find by \cref{lem:perturbation_removal} that 
there exists a constant $C>0$ such that 
\begin{equation}\label{thm:criterion_irrationality_original:perturbation_removal}
\#\mathcal{F}_{2}
\coloneqq 
\#\{
N\in[1,\eta x_{n})\cap\mathbb{Z}
\mid
\xi_{N}(\bm{a})>\xi_{N}(|\bm{b}|)
\}
>
Cx_n
\end{equation}
holds for every large integer $n$. Thus, for the set 
\[
\mathcal{F}
\coloneqq
\mathcal{F}_{1}\cap\mathcal{F}_{2}
=
\{
N\in[1,\eta x_{n})\cap\mathbb{Z}
\mid
\xi_{N}(\bm{a})>\xi_{N}(|\bm{b}|),\ 
y_n^{d-1}\xi_{N}(|\bm{a}+\bm{b}|)<\epsilon
\},
\] 
by using
\cref{thm:criterion_irrationality_original:small_tail} and
\cref{thm:criterion_irrationality_original:perturbation_removal}
together with $\#(\mathcal{F}_1\cup \mathcal{F}_2)\le \eta x_n$, 
we obtain  
\[
\#\mathcal{F}
=
\#\mathcal{F}_{1}
+
\#\mathcal{F}_{2}
-
\#(\mathcal{F}_{1}\cup\mathcal{F}_{2})
>
Cx_n+o(x_n)
\quad\text{as $n\to\infty$},
\]
and hence $\#\mathcal{F}\to\infty$ by condition~\cref{thm:criterion_irrationality_original:x_infty}. 
Let $n$ be a large integer and $N=\max\mathcal{F}$. 
Then, we have $N\to \infty$ as $n\to\infty$
since $\#\mathcal{F}\to\infty$ as $n\to\infty$, and the inequalities 
\[
\xi_{N}(\bm{a})>\xi_{N}(|\bm{b}|)
\quad\text{and}\quad 
y_n^{d-1}\xi_{N}(|\bm{a}+\bm{b}|)<\epsilon
\]  
are satisfied. 
Hence, by condition~\cref{thm:criterion_irrationality_original:S_house_bound}, 
we obtain 
\begin{align}
\xi_{N}(|\bm{a}+\bm{b}|)\cdot S_{\bm{a}+\bm{b}}(N)^{d-1}
&\le
\epsilon y_{n}^{-(d-1)}
\cdot
\left(
S_{\bm{a}}(x_n)+S_{\bm{b}}(x_{n})
\right)^{d-1}\\
&=
O(\epsilon y_{n}^{-(d-1)}\cdot y_{n}^{d-1})\quad\text{as $n\to\infty$}\\
&=
O(\epsilon)\quad\text{as $n\to\infty$},
\end{align}
and therefore, \cref{thm:criterion_irrationality_original} follows
by applying \cref{lem:fundamental_criterion}.
\end{proof}

\section{Proof of \cref{thm:criterion_irrationality_prepared}}
\label{subsec:deduction_prepared}
We use the following lemma for the proof of \cref{thm:criterion_irrationality_prepared}. 

\begin{lemma}
\label{lem:use_summatory_function} 
Let $\bm{c}$
be a sequence of complex numbers. 
Suppose that there exist sequences
$(x_n)_{n=1}^{\infty}$, $(w_{n})_{n=1}^{\infty}$, $(z_{n})_{n=1}^{\infty}$
of real numbers $\ge1$ and a constant $\eta\in(0,1]$ 
satisfying the following conditions as $n\to\infty$\textup{:}
\begin{enumerate}
\item\label{lem:use_summatory_function:moderate_growth}
$\sum_{m<x_n}|c(m)|
=
o(q^{z_{n}}x_{n}/w_{n})$.
\item\label{lem:use_summatory_function:shifted_decay}
$R_{\bm{c}}(\eta x_{n},(1-\eta)x_{n})
=
o(x_{n}/w_{n})$.
\end{enumerate}
Then we have
\begin{equation}
R_{\bm{c}}(\eta x_{n},z_{n})
=
o(x_{n}/w_{n})
\quad\text{as $n\to\infty$}.
\end{equation}
\end{lemma}
\begin{proof}
Decompose $R_{\bm{c}}(\eta x_{n},z_{n})$ as
\begin{equation}
\label{lem:use_summatory_function:decomp}
R_{\bm{c}}(\eta x_{n},z_{n})
=
\sum_{N<\eta x_{n}}
\sum_{j\ge z_{n}}
\frac{|c(N+j)|}{q^{j}}
=
R_{1}+R_{2},
\end{equation}
where 
\[
R_{1}
\coloneqq
\sum_{N<\eta x_{n}}
\sum_{z_{n}\le j<x_{n}-N}
\frac{|c(N+j)|}{q^{j}}
\quad\text{and}\quad 
R_{2}
\coloneqq 
\sum_{N<\eta x_{n}}
\sum_{j\ge x_{n}-N}
\frac{|c(N+j)|}{q^{j}}.
\]
Then, by conditions~\cref{lem:use_summatory_function:moderate_growth} and
\cref{lem:use_summatory_function:shifted_decay},
we have 
\[
R_{1}
\le
\sum_{j\ge z_{n}}
q^{-j}
\sum_{N<x_{n}-j}
|c(N+j)|
\le
\sum_{j\ge z_{n}}
q^{-j}
\sum_{n<x_n}|c(n)|
=
o(x_n/w_n)
\quad\text{as $n\to\infty$}
\]
and
\[
R_{2}
\le
\sum_{N<\eta x_{n}}
\sum_{j\ge(1-\eta)x_{n}}
\frac{|c(N+j)|}{q^{j}}
=
R_{\bm{c}}(\eta x_{n},(1-\eta)x_{n})=o(x_n/w_n)\quad\text{as $n\to\infty$},
\]
respectively. Thus \cref{lem:use_summatory_function} is proved. 
\end{proof}

\begin{proof}[Proof of \cref{thm:criterion_irrationality_prepared}.]
Let $\bm{a}$, $\bm{b}$,
$(x_n)_{n=1}^{\infty}$, $(y_n)_{n=1}^{\infty}$, $(z_n)_{n=1}^{\infty}$
and $\rho$ be as in \cref{thm:criterion_irrationality_prepared}. 
Then, all conditions in \cref{thm:criterion_irrationality_original}
other than \cref{thm:criterion_irrationality_original:average_decay} are clearly satisfied 
from the assumptions of \cref{thm:criterion_irrationality_prepared}. 
In what follows, we check condition \cref{thm:criterion_irrationality_original:average_decay}
in \cref{thm:criterion_irrationality_original} by using \cref{lem:use_summatory_function}. 
We first note 
\begin{equation}
\label{thm:criterion_irrationality_prepared:def_lambda_rho_lower_bound}
\lambda
\coloneqq
\varlimsup_{n\to\infty}{y_n}^{(d-1)/x_n}
\ge
1
\quad\text{and}\quad 
\rho
\ge
\varlimsup_{n\to\infty}\house{a(n)}^{1/n}
\ge
1,
\end{equation}
which follow by $y_n\ge1$ for $n\ge1$
and condition~\cref{thm:criterion_irrationality_prepared:pointwise};
indeed, $\house{a(n)}\geq1$ for infinitely many $n$ 
since $\bm{a}$ is a sequence of algebraic integers 
with $\#\mathcal{N}_{\bm{a}}=\infty$.  
Fix a constant $\mu>1$ with $\lambda<\mu<q/\rho$,
which is possible since $1\le \lambda<q/\rho$
by \cref{thm:criterion_irrationality_prepared:def_lambda_rho_lower_bound}
and condition \cref{thm:criterion_irrationality_prepared:yn_bound}.  
Then, there exists a constant $\delta\in(0,1)$ small enough satisfying   
\begin{equation}\label{thm:criterion_irrationality_prepared:delta}
\lambda<\mu^{1-\delta}<\frac{q^{1-\delta}}{\rho}.
\end{equation}
Hence, by
\cref{thm:criterion_irrationality_prepared:def_lambda_rho_lower_bound} and
\cref{thm:criterion_irrationality_prepared:delta},
we have $y_n^{(d-1)/x_n}<\mu^{1-\delta}$ for every large $n$, 
and so there exists a constant $c_1>0$ such that 
\begin{equation}\label{thm:criterion_irrationality_prepared:yn_mu_ineq}
y_{n}^{d-1}
\le 
c_1\mu^{(1-\delta)x_{n}}
\quad\text{for $n\geq1$},
\end{equation}
where $c_1$ (and $c_2,c_3$ below) is independent of $n$. 
Moreover, by
\cref{thm:criterion_irrationality_prepared:def_lambda_rho_lower_bound} and
\cref{thm:criterion_irrationality_prepared:delta},  
\begin{equation}\label{thm:criterion_irrationality_prepared:nu}
1
\le
\rho
<
\nu
\coloneqq
\rho^{1/2}
\left(
\frac{q}{\mu}
\right)^{(1-\delta)/2}
<
\left(
\frac{q}{\mu}
\right)^{1-\delta}
<
q,
\end{equation}
where the last inequality follows from $\delta<1<\mu,q$. 
Thus, by
\cref{thm:criterion_irrationality_prepared:def_lambda_rho_lower_bound} and
\cref{thm:criterion_irrationality_prepared:nu},
there exists a constant $c_2>0$ such that $a(n)\le c_{2}\nu^{n}$ for $n\ge 1$. 
Therefore,
by \cref{thm:criterion_irrationality_prepared:delta},
\cref{thm:criterion_irrationality_prepared:yn_mu_ineq} and
\cref{thm:criterion_irrationality_prepared:nu},
we obtain 
\begin{align}
R_{\bm{a}}(\delta x_n,(1-\delta)x_{n})
=
\sum_{N<\delta x_n}
\sum_{j\ge (1-\delta)x_n}
\frac{a(N+j)}{q^j}
&\le 
c_2
\sum_{N<\delta x_n}
\sum_{j\ge(1-\delta)x_n}
\frac{\nu^{N+j}}{q^j}\\
&\le 
c_2\delta x_n
\sum_{j\ge(1-\delta)x_n}
\frac{\nu^{\delta x_n+j}}{q^j}\\
&\le 
c_3
\delta x_n 
\left(
\frac{\nu}{q^{1-\delta}}
\right)^{x_n}\\
&=
c_3\delta x_n\cdot o(\mu^{-(1-\delta)x_n})
\quad\text{as $n\to\infty$}\\
&=
o(x_n /y_{n}^{d-1})
\quad\text{as $n\to\infty$},
\end{align}
which shows that condition~\cref{lem:use_summatory_function:shifted_decay}
in \cref{lem:use_summatory_function} is satisfied for $\bm{c}\coloneqq\bm{a}$,
$(w_n)_{n=1}^{\infty}\coloneqq (y_n^{d-1})_{n=1}^{\infty}$ and $\eta\coloneqq\delta$. 
Moreover, condition~\cref{lem:use_summatory_function:moderate_growth}
in \cref{lem:use_summatory_function} is satisfied
by condition~\cref{thm:criterion_irrationality_prepared:moderate_growth}.
Hence, by \cref{lem:use_summatory_function}, we have 
\[
R_{\bm{a}}(\delta x_n,z_n)=o(x_n/y_n^{d-1})\quad\text{as $n\to\infty$}.
\]
In the same way,
we can derive $R_{\bm{b}}(\delta x_n,z_n)=o(x_n/y_n^{d-1})$ as $n\to\infty$, 
and therefore, condition~\cref{thm:criterion_irrationality_original:average_decay}
in \cref{thm:criterion_irrationality_original} is satisfied for $\eta\coloneqq\delta$. 
Thus, applying \cref{thm:criterion_irrationality_original},
we obtain \cref{thm:criterion_irrationality_prepared}.
\end{proof}

\section{Proofs of \cref{cor:MinkowskiSum,cor:degree_ell_irrationality}}\label{sec:addcor}
\begin{proof}[Proof of \cref{cor:MinkowskiSum}] 
We apply \cref{thm:criterion_irrationality_prepared}. 
Let $(x_n)_{n=1}^{\infty}$ and $(y_n)_{n=1}^{\infty}$
be sequences of real numbers $\ge1$ defined, respectively, by 
\begin{equation}\label{cor:MinkowskiSum:def_xn_yn}
x_n\coloneqq n\quad\text{for $n\ge1$}
\quad\text{and}\quad 
y_n\coloneqq 
\left\{
\begin{array}{cl}
1&\text{for $n=1$},\\
\max(S_{\bm{u}}(n),S_{\bm{v}}(n))&\text{for $n\ge2$}.
\end{array}
\right.
\end{equation}
Since $u(n)>0$ for $n\ge 1$, we have $\house{u(n)}\geq 1$ for $n\ge 1$ and so 
\begin{equation}\label{cor:MinkowskiSum:yn_lower_bound}
y_n
\ge
S_{\bm{u}}(n)
=
\sum_{m<n}\house{u(m)}
\ge
n-1
\quad\text{for $n\ge 2$}. 
\end{equation}
Define 
\begin{equation}\label{cor:MinkowskiSum:def_wn}
w_n
\coloneqq
\#\bigl((\mathcal{A}+\mathcal{B})\cap [1,n)\bigr)\cdot \frac{\log y_n}{n}
\quad\text{for $n\ge 1$}. 
\end{equation}
Then, by \cref{cor:MinkowskiSum:yn_lower_bound} and $\#\mathcal{A}=\infty$,
there exists a positive integer $n_0$ such that 
$w_n>0$ for every integer $n$ with $n> n_0$. 
Let $(z_n)_{n=1}^{\infty}$ be a sequence of real numbers $\ge1$
defined by
\[
z_{n}
\coloneqq
1\quad\text{for $1\le n\le n_0$}
\quad\text{and}\quad
z_{n}
\coloneqq
\max(1,w_n^{-1/2}\log y_n)
\quad\text{for $n>n_{0}$}.
\]
By \cref{cor:MinkowskiSum:sparsity} and \cref{cor:MinkowskiSum:yn_lower_bound},
we have
\begin{equation}
\label{cor:MinkowskiSum:wn_yn_limit}
w_n\to0
\quad\text{and}\quad
y_n\to\infty
\quad\text{as $n\to\infty$}.
\end{equation}
Hence, 
\begin{equation}\label{cor:MinkowskiSum:zn_reduction}
z_{n}
=
w_{n}^{-1/2} \log y_n
\end{equation}
hold for all large $n$ and so
\begin{align}\label{cor:MinkowskiSum:zn_log_yn_infty}
\lim_{n\to\infty}
\frac{z_n}{\log y_n}
=
\lim_{n\to\infty}
w_n^{-1/2}
=
\infty.
\end{align}

Now we define sequences
$\bm{a}=(a(n))_{n=1}^{\infty}$ and $\bm{b}=(b(n))_{n=1}^{\infty}$ by 
\[
a(n)
\coloneqq 
\left\{
\begin{array}{cl}
u(n)&\text{if $n\in\mathcal{A}+\mathcal{B}$},\\
0&\text{otherwise},
\end{array}
\right. 
\quad\text{and}\quad
b(n)\coloneqq 
\left\{
\begin{array}{cl}
v(n)&\text{if $n\in\mathcal{B}$},\\
0&\text{otherwise},
\end{array}
\right.  
\]
respectively.
Then, $a(n)\ge 0$ for $n\ge 1$ and
$\#\mathcal{N}_{\bm{a}}=\infty$ since $u(n)>0$ for $n\ge1$ and $\# \mathcal{A}=\infty$. 
In what follows, we check conditions
\cref{thm:criterion_irrationality_original:x_infty},
\cref{thm:criterion_irrationality_original:S_house_bound},
\cref{thm:criterion_irrationality_original:sparsity},
\cref{thm:criterion_irrationality_original:interlace}
in \cref{thm:criterion_irrationality_original} and
conditions
\cref{thm:criterion_irrationality_prepared:pointwise},
\cref{thm:criterion_irrationality_prepared:yn_bound},
\cref{thm:criterion_irrationality_prepared:moderate_growth}
in \cref{thm:criterion_irrationality_prepared}. 
Conditions
\cref{thm:criterion_irrationality_original:x_infty} and
\cref{thm:criterion_irrationality_original:S_house_bound}
in \cref{thm:criterion_irrationality_original} are clear by \cref{cor:MinkowskiSum:def_xn_yn}. 
Moreover, condition~\cref{thm:criterion_irrationality_original:sparsity}
in \cref{thm:criterion_irrationality_original} is satisfied since,
by
\cref{cor:MinkowskiSum:def_wn},
\cref{cor:MinkowskiSum:wn_yn_limit} and
\cref{cor:MinkowskiSum:zn_reduction},
we have  
\begin{equation}\label{cor:MinkowskiSum:a_density_zero}
\#\mathcal{N}_{\bm{a}}(x_n)
=
\#\bigl((\mathcal{A}+\mathcal{B})\cap[1,n)\bigr)
=
w_n\frac{n}{\log y_n}
=
w_n^{1/2}\frac{n}{z_n}
=
o(n/z_n)
\quad\text{as $n\to\infty$}
\end{equation}
and, by writing $\ell_{\mathcal{A}}\coloneqq\min_{\ell\in\mathcal{A}}\ell$, we have
\begin{align}
\#\mathcal{N}_{\bm{b}}(x_n)
=
\sum_{\substack{
m<n\\
m\in\mathcal{B}
}}
1
\le
\sum_{\substack{
m<n\\
\ell_{\mathcal{A}}+m\in\mathcal{A}+\mathcal{B}
}}
1
=
\sum_{\substack{
m<n+\ell_{\mathcal{A}}\\
m\in\mathcal{A}+\mathcal{B}
}}
1
&\le
\sum_{\substack{
m<n\\
m\in\mathcal{A}+\mathcal{B}
}}
1
+
\sum_{n\le n<n+\ell_{\mathcal{A}}}
1\\
&\le
\#\bigl((\mathcal{A}+\mathcal{B})\cap[1,n)\bigr)
+
\ell_{\mathcal{A}}\\
&\le
2\#\mathcal{N}_{\bm{a}}(x_n)
=
o(n/z_n)
\quad\text{as $n\to\infty$}.
\end{align}
Next we check condition~\cref{thm:criterion_irrationality_original:interlace}
in \cref{thm:criterion_irrationality_original}. 
Assume $\#\mathcal{N}_{\bm{b}}=\infty$
and let $\Delta,L>1$ be constants as in \cref{cor:MinkowskiSum:interlace}.
Since $\mathcal{N}_{\bm{b}}\subset\mathcal{B}$,
we then have $\#\mathcal{B}=\infty$, so we can use \cref{cor:MinkowskiSum:interlace}.
Thus, for any consecutive indices $m<m_{+}$ in $\mathcal{N}_{\bm{b}}$ and 
any real number $\mu\ge L$ with $m+\Delta\mu<m_{+}$, 
there exists an integer $\ell\in\mathcal{A}\cap[\mu,\Delta\mu)$ by \cref{cor:MinkowskiSum:interlace} 
and so we obtain 
\[
\ell+m
\in 
(\mathcal{A}+\mathcal{B})\cap[m+\mu,m+\Delta\mu)
=
\mathcal{N}_{\bm{a}}\cap[m+\mu,m+\Delta\mu),
\]
which shows that condition~\cref{thm:criterion_irrationality_original:interlace}
in \cref{thm:criterion_irrationality_original} is fulfilled.  
Conditions
\cref{thm:criterion_irrationality_prepared:pointwise} and
\cref{thm:criterion_irrationality_prepared:yn_bound}
in \cref{thm:criterion_irrationality_prepared} are satisfied
by $\house{a(n)},\house{b(n)}\le y_{n}$
and $\lim_{n\to\infty} y_n^{1/n}=1$,
the latter of which follows immediately from \cref{cor:MinkowskiSum:wn_yn_limit}
and \cref{cor:MinkowskiSum:a_density_zero} since 
\[
\log y_{n}^{1/n}
=
\frac{1}{n}\log y_{n}
=
\frac{w_n}{\#\mathcal{N}_{\bm{a}}(x_n)}
\to
0
\quad\text{as $n\to\infty$}.
\]
We also find by \cref{cor:MinkowskiSum:zn_log_yn_infty} that 
\[
\sum_{m<n}a(m),
\sum_{m<n}|b(m)|\leq y_n
=
o(q^{z_{n}}n/y_{n}^{d-1})
\quad\text{as $n\to\infty$}
\]
holds for any fixed positive integer $d$, and hence
condition \cref{thm:criterion_irrationality_prepared:moderate_growth}
in \cref{thm:criterion_irrationality_prepared} is fulfilled. 
Thus, \cref{cor:MinkowskiSum} follows by applying \cref{thm:criterion_irrationality_prepared}.  
\end{proof}

To prove \cref{cor:degree_ell_irrationality},
we first prove the following Liouville-type transcendence criterion.
Note that we do not assume that $q$ is a Pisot or Salem number,
nor the condition \cref{cor:degree_ell_irrationality:cond}.
\begin{lemma}
\label{lem:Liouville}
Let $q>1$ be a real algebraic integer, $K$ be a finite extension of $\mathbb{Q}(q)$
and $(a(n))_{n=1}^{\infty}$ be a sequence of algebraic integers of $K$
such that $a(n)\geq0$ for $n\ge1$ and $\#\mathcal{N}_{\bm{a}}=\infty$. 
Write the elements of $\mathcal{N}_{\bm{a}}$
as $n_1,n_2,\dots$ so that $(n_{k})_{k=1}^{\infty}$ is strictly increasing. 
Suppose that
\begin{equation}
\label{lem:Liouville:coefficient_bound}
\varlimsup_{k\to\infty}\house{a(n_{k})}^{\frac{1}{n_{k}}}<q
\end{equation}
and
\begin{equation}
\label{lem:Liouville:lacunary}
\varlimsup_{k\to\infty}
\frac{n_{k+1}}{n_{k}}
=\infty.
\end{equation}
Then, the number $\sum_{n=1}^{\infty}a(n)/{q^{n}}$ is transcendental.
\end{lemma}
\begin{proof}
By \cref{lem:Liouville:coefficient_bound},
we can take $\delta>0$ such that
\begin{equation}
\label{lem:Liouville:Q_bound}
\house{a(n_{k})}\ll q^{(1-\delta)n_{k}}\ll\house{q}^{(1-\delta)n_{k}}
\quad\text{as $k\to\infty$}.
\end{equation}
Assume to the contrary that $\xi\coloneqq\xi(\bm{a})$ is an algebraic number
and let $\mathcal{O}_{L}$ be the ring of integers of $L\coloneqq K(\xi)$.
Then, we can write $\xi=\gamma/u$
with some $\gamma\in\mathcal{O}_{L}$ and $u\in\mathbb{Z}_{\ge1}$.
Let $N\ge1$ be a sufficiently large integer.
We then have
\begin{equation}
\label{lem:Liouville:OL}
0
<
u\xi_{1+n_{N}}(\bm{a})
=
q^{1+n_{N}}\gamma
-
u\sum_{k\le N}a(n_{k})q^{1+n_{N}-n_{k}}
\in
\mathcal{O}_{L}.
\end{equation}
Hence, by taking the norm with respect to the extension $L/\mathbb{Q}$, we obtain
\begin{equation}
\label{lem:Liouville:norm}
|N_{L/\mathbb{Q}}(u\xi_{1+n_{N}}(\bm{a}))|\ge1.
\end{equation}
By \cref{lem:Liouville:Q_bound}, we have  
\begin{equation}
\label{lem:Liouville:error}
u\xi_{1+n_{N}}(\bm{a})
=
uq^{1+n_{N}}\sum_{k>N}\frac{a(n_k)}{q^{n_k}}
\ll
uq^{1+n_{N}}\sum_{k>N}q^{-\delta n_{k}}
\ll
q^{n_{N}-\delta n_{N+1}}
\ll
\house{q}^{n_{N}}q^{-\delta n_{N+1}},
\end{equation}
where, hereafter in this paragraph, the implicit constants are independent of $N$.
For any embedding $\sigma\colon L\to\mathbb{C}$,
by \cref{lem:Liouville:Q_bound}
and \cref{lem:Liouville:OL}, we have 
\begin{equation}
\label{lem:Liouville:error_conjugate}
|(u\xi_{1+n_{N}}(\bm{a}))^{\sigma}|
\ll
\house{q}^{n_{N}}+\sum_{k\le N}\house{a(n_{k})}\house{q}^{n_{N}-n_{k}}
\ll
\house{q}^{n_{N}}
\biggl(1
+
\sum_{k\le N}\house{q}^{-\delta n_{k}}
\biggr)
\ll
\house{q}^{n_{N}}.
\end{equation}
Hence, by writing $D\coloneqq[L:\mathbb{Q}]$ and using
\cref{lem:Liouville:lacunary},
\cref{lem:Liouville:norm},
\cref{lem:Liouville:error}
and \cref{lem:Liouville:error_conjugate},
there are arbitrary large positive integer $N$ satisfying
$n_{N+1}/n_{N}\ge\frac{2D\log\house{q}}{\delta\log q}$ so
\[
1
\ll
\house{q}^{n_{N}}q^{-\delta n_{N+1}}
\cdot
(\house{q}^{n_{N}})^{D-1}
=
\house{q}^{Dn_{N}}q^{-\delta n_{N+1}}
=
q^{\frac{D\log\house{q}}{\log q}n_{N}-\delta n_{N+1}}
\le
q^{-\frac{\delta}{2}n_{N+1}},
\] 
which is a contradiction for large $N$. Thus, $\xi(\bm{a})$ is transcendental.
\end{proof}

\begin{proof}[Proof of \cref{cor:degree_ell_irrationality}]
By \cref{cor:degree_ell_irrationality:coefficient_bound},
we can take $\delta>0$ such that
\begin{equation}
\label{cor:degree_ell_irrationality:delta}
\house{a(n_{k})}\ll q^{(1-\delta)n_{k}}
\quad\text{as $k\to\infty$}.
\end{equation}
Also, note that $Q(n)$ is non-decreasing.
By \cref{lem:Liouville}, we may assume
\begin{equation}
\label{cor:degree_ell_irrationality:ratio_finite}
\varlimsup_{k\to\infty}
\frac{n_{k+1}}{n_{k}}
<
\infty.
\end{equation}
Suppose to the contrary that $\xi\coloneqq\xi(\bm{a})$
is algebraic over $\mathbb{Q}(q)$ of degree $\le\ell$.
Then there exist algebraic integers 
$u_0>0,u_1,\dots,u_\ell$ of $\mathbb{Q}(q)$ such that 
\begin{equation}
\label{cor:degree_ell_irrationality:algebraic_eq}
u_{0}\xi^{\ell}+u_{1}\xi^{\ell-1}+\cdots+u_{\ell}=0.
\end{equation}
For integers $j=1,2,\dots,\ell$, we put   
\begin{equation}\label{cor:degree_ell_irrationality:expansion}
\xi^{j}
=
\biggl(
\sum_{k=1}^{\infty}
\frac{a(n_{k})}{q^{n_{k}}}
\biggr)^{j}
=
\sum_{n=1}^{\infty}
\frac{b_{j}(n)}{q^{n}},
\end{equation}
where
\begin{equation}\label{cor:degree_ell_irrationality:def_bj}
b_{j}(n)
\coloneqq
\sum_{\substack{
(n_{k_1},n_{k_2},\dots, n_{k_j})\in\mathcal{N}_{\bm{a}}^j\\
n_{k_1}+n_{k_2}+\cdots+n_{k_j}=n
}}
a(n_{k_1})a(n_{k_2})\cdots a(n_{k_j}).
\end{equation}
Let $(\alpha(n))_{n=1}^{\infty}$ and $(\beta(n))_{n=1}^{\infty}$
be sequences of algebraic integers of $\mathbb{Q}(q)$ defined by   
\begin{equation}
\label{cor:degree_ell_irrationality:def_alpha_n}
\alpha(n)
\coloneqq
u_{0}b_{\ell}(n)
\quad\text{for $n\geq1$}
\end{equation}
and 
\begin{equation}
\label{cor:degree_ell_irrationality:def_beta_n}
\beta(n)
\coloneqq
u_{1}b_{\ell-1}(n)+\cdots+u_{\ell-1}b_{1}(n)
\quad\text{for $n\ge 1$},
\end{equation}
respectively. 
Then, by \eqref{cor:degree_ell_irrationality:algebraic_eq}
and \eqref{cor:degree_ell_irrationality:expansion}, we have   
\begin{equation}\label{cor:degree_ell_irrationality:derived_series}
\sum_{n=1}^{\infty}
\frac{\alpha(n)+\beta(n)}{q^{n}}
=
u_{0}\xi^{\ell}+u_{1}\xi^{\ell-1}+\cdots+u_{\ell-1}\xi
=
-u_{\ell}\in\mathbb{Q}(q). 
\end{equation}

Now we apply \cref{thm:criterion_irrationality_prepared} and deduce a contradiction. 
Clearly, by \cref{cor:degree_ell_irrationality:def_bj}
and \cref{cor:degree_ell_irrationality:def_alpha_n},
we have $\alpha(n)\geq0$ for $n\ge 1$ and $\#\mathcal{N}_{\bm{\alpha}}=\infty$;
in particular, $\house{\alpha(n)}\ge 1$ for infinitely many integers $n$. 
Moreover, by the definition of $Q(n)$ and \cref{cor:degree_ell_irrationality:delta}, we have
\begin{equation}
\label{cor:degree_ell_irrationality:bj_bound}
\begin{aligned}
\house{b_{j}(n)}
&\ll
\sum_{\substack{
(n_{k_1},n_{k_2},\dots, n_{k_j})\in\mathcal{N}_{\bm{a}}^j\\
n_{k_1}+n_{k_2}+\cdots+n_{k_j}=n
}}
\min\Bigl(q^{(1-\delta)(n_{k_1}+\cdots+n_{k_j})},
Q(n_{k_{1}})\cdots Q(n_{k_{j}})\Bigr)\\
&\ll
n^{j}
\min(q^{(1-\delta)n},Q(n)^{j})
\ll
\min(n^{\ell}q^{(1-\delta)n},Q(n)^{c_{1}})
\quad\text{as $n\to\infty$}
\end{aligned}
\end{equation}
for each $j=1,2,\dots,\ell$,
where, hereafter in this proof,
$c_{1},c_{2},\ldots$ are appropriate positive constants independent of $j$, $k$ and $n$.
Hence, by \cref{cor:degree_ell_irrationality:def_alpha_n},
\cref{cor:degree_ell_irrationality:def_beta_n} and
\cref{cor:degree_ell_irrationality:bj_bound}, we obtain 
\begin{equation}
\label{cor:degree_ell_irrationality:alpha_beta_bound}
\house{\alpha(n)},\house{\beta(n)}
\le
Q(n)^{c_{2}}
\quad\text{for $n\ge1$}.
\end{equation}
Also, by \cref{cor:degree_ell_irrationality:bj_bound}, we have
\begin{equation}
\label{cor:degree_ell_irrationality:pointwise}
\rho
\coloneqq
\varlimsup_{n\to\infty}
\max(\house{\alpha(n)},\house{\beta(n)})^{1/n}
\le
q^{1-\delta}<q.
\end{equation}
On the other hand, by \cref{cor:degree_ell_irrationality:cond},
there exists a strictly increasing sequence $(s_k)_{k=1}^{\infty}$ of positive integers
such that
\begin{equation}
\label{cor:degree_ell_irrationality:sk_cond_pre} 
\lim_{k\to\infty}
\frac{n_{s_k}}{{s_k}^{\ell}\log Q(n_{s_k})}
=\infty. 
\end{equation}  
Let $(x_k)_{k=1}^{\infty}$ be a sequence of positive integers defined by    
\[
x_k\coloneqq n_{s_k}
\quad\text{for $k\ge 1$}
\]
so that \cref{cor:degree_ell_irrationality:sk_cond_pre} takes the form
\begin{equation}
\label{cor:degree_ell_irrationality:sk_cond} 
\lim_{k\to\infty}
\frac{x_{k}}{s_{k}^{\ell}\log Q(x_{k})}
=\infty. 
\end{equation}  
Then, by \cref{cor:degree_ell_irrationality:alpha_beta_bound}, we obtain 
\begin{equation}
\label{cor:degree_ell_irrationality:S_house_bound}
S_{\bm{\alpha}}(x_k),S_{\bm{\beta}}(x_k)
\le
\sum_{n<x_k}(\house{\alpha(n)}+\house{\beta(n)})
=
O(x_{k}Q(x_{k})^{c_{2}})
=
O(Q(x_{k})^{c_{3}})
\quad\text{as $k\to\infty$}.
\end{equation}
Let $(y_k)_{k=1}^{\infty}$ and $(z_k)_{k=1}^{\infty}$ be sequences of real numbers defined by    
\begin{equation}
\label{cor:degree_ell_irrationality:def_yk}
y_{k}\coloneqq Q(x_{k})^{c_{3}}
\quad\text{for $k\ge 1$}
\end{equation}
and 
\begin{equation}
\label{cor:degree_ell_irrationality:def_zk}
z_{k}\coloneqq \lambda_{k}\log Q(x_{k})
\quad\text{for $k\ge 1$},
\end{equation}
where 
\begin{equation}
\label{cor:degree_ell_irrationality:def_lambdak}
\lambda_{k}
\coloneqq
\left(\frac{x_{k}}{s_{k}^{\ell}\log Q(x_{k})}\right)^{1/2}
\quad\text{for $k\ge1$}.
\end{equation}
Note that $\lambda_{k}\to\infty$ as $k\to\infty$ by \cref{cor:degree_ell_irrationality:sk_cond}.  
Then, by \cref{cor:degree_ell_irrationality:pointwise} and \cref{cor:degree_ell_irrationality:sk_cond}, we have
\begin{equation}
\label{cor:degree_ell_irrationality:yk_bound}
\varlimsup_{k\to\infty}y_{k}^{(d-1)/x_{k}}=1<q^{\delta}\le q/\rho, 
\end{equation}
and, by
\cref{cor:degree_ell_irrationality:def_bj},
\cref{cor:degree_ell_irrationality:sk_cond},
\cref{cor:degree_ell_irrationality:def_zk} and
\cref{cor:degree_ell_irrationality:def_lambdak}, we have 
\begin{equation}\label{cor:degree_ell_irrationality:sparsity_alpha}
\#\mathcal{N}_{\bm{\alpha}}(x_k)
\le 
\#
\{
(n_{k_1},\dots,n_{k_\ell})\in \mathcal{N}_{\bm{a}}^{\ell}
\mid
n_{k_1}+\cdots+n_{k_\ell}\le x_k=n_{s_k}
\}
\le
s_k^{\ell}
=
o({x_k}/{z_k})
\end{equation}
and  
\begin{equation}\label{cor:degree_ell_irrationality:sparsity_beta}
\#\mathcal{N}_{\bm{\beta}}(x_k)
\le
\sum_{j=1}^{\ell-1}
\#\{
(n_{k_1},\dots,n_{k_j})\in\mathcal{N}_{\bm{a}}^{j}
\mid
n_{k_1}+\cdots+n_{k_{j}}\le x_k=n_{s_k}
\}
\le
s_k^{\ell}
=
o({x_k}/{z_k})
\end{equation}
as $k\to\infty$. 
Moreover, by
\cref{cor:degree_ell_irrationality:def_yk} and
\cref{cor:degree_ell_irrationality:def_zk}, we have 
\[
q^{z_{k}}x_{k}/y_{k}^{d-1}
\ge
Q(x_{k})^{\lambda_k\log q-c_{3}(d-1)}
\]
and so, by \cref{cor:degree_ell_irrationality:S_house_bound}, 
\[
\sum_{n<x_k}\alpha(n),
\sum_{n<x_k}|\beta(n)|
=
O(Q(x_k)^{c_{3}})
=
o(q^{z_k}x_k/y_{k}^{d-1})
\quad\text{as $k\to\infty$}
\]
since $\lambda_k\to\infty$ as $k\to\infty$. 
Thus, by the above argument, we find that conditions 
\cref{thm:criterion_irrationality_original:x_infty,%
thm:criterion_irrationality_original:S_house_bound,%
thm:criterion_irrationality_original:sparsity}
in \cref{thm:criterion_irrationality_original} and
\cref{thm:criterion_irrationality_prepared:pointwise,%
thm:criterion_irrationality_prepared:yn_bound,%
thm:criterion_irrationality_prepared:moderate_growth}
in \cref{thm:criterion_irrationality_prepared} are satisfied.  

Finally, we check condition \cref{thm:criterion_irrationality_original:interlace}
in \cref{thm:criterion_irrationality_original}. 
We assume $\#\mathcal{N}_{\bm{\beta}}=\infty$, then $\ell\geq2$. 
Define 
\[
L\coloneqq 1+(\ell-1)n_1.
\] 
By \cref{cor:degree_ell_irrationality:ratio_finite},
there exists a constant $\Delta>1$ such that 
\begin{equation}
\label{cor:degree_ell_irrationality:Delta_cond}
n_{k+1}<\Delta n_k
\quad\text{for $k\ge 1$}. 
\end{equation}
Let $m$ and $m_+$ be consecutive integers in $\mathcal{N}_{\bm{\beta}}$.  
Since $\beta(m)\neq0$,
there exists an integer $j_{0}\in\{1,\ldots,\ell-1\}$ such that $b_{j_0}(m)\neq0$, 
and so there exists $(n_{k_1},n_{k_2},\dots,n_{k_{j_0}})\in\mathcal{N}_{\bm{a}}^{j_0}$
such that 
\begin{equation}
\label{cor:degree_ell_irrationality:beta_support_eq}
n_{k_1}+n_{k_2}+\cdots+n_{k_{j_0}}=m.
\end{equation}
Then, for a fixed real number $\mu\ge L>(\ell-1)n_1$ with $m+\Delta \mu<m_+$,
there exist consecutive integers $n_{\kappa},n_{\kappa+1}$ in $\mathcal{N}_{\bm{a}}$
such that 
\begin{equation}\label{cor:degree_ell_irrationality:kappa_cond}
(\ell-j_{0})n_{\kappa}<\mu\le (\ell-j_{0})n_{\kappa+1}.
\end{equation}
Let $h\coloneqq m+(\ell-j_{0})n_{\kappa+1}$.
By \cref{cor:degree_ell_irrationality:Delta_cond} and
\cref{cor:degree_ell_irrationality:kappa_cond}, we then have
\[
m+\mu
\le
h
<
m+(\ell-j_{0})\Delta n_\kappa
<
m+\Delta\mu
\]
and, by \cref{cor:degree_ell_irrationality:beta_support_eq},
the integer $h$ is expressed as the sum of $\ell$ integers in $\mathcal{N}_{\bm{a}}$: 
\[
h
=
n_{k_1}+n_{k_2}+\cdots+n_{k_{j_0}}
+
\underbrace{n_{\kappa+1}+\cdots+n_{\kappa+1}}_{\text{$\ell-j_{0}$ terms}}.
\]
Therefore, by
\cref{cor:degree_ell_irrationality:def_bj} and
\cref{cor:degree_ell_irrationality:def_alpha_n},
we have $\alpha(h)>0$ and so $h\in\mathcal{N}_{\bm{\alpha}}\cap[m+\mu,m+\Delta\mu)$. 

Thus, applying \cref{thm:criterion_irrationality_prepared},
we find that the infinite series \cref{cor:degree_ell_irrationality:derived_series}
does not belong to $\mathbb{Q}(q)$. 
This is a contradiction. 
\end{proof}

\section{Proofs of \cref{prethm:ErdosLemma4prime} and \cref{cor:chimera}}
\label{sec:proof_corollary}
We first complete the proof of Erd\H{o}s's \cref{prethm:ErdosLemma4prime}.
\begin{proof}[Proof of \cref{prethm:ErdosLemma4prime}]
Let $t$, $\bm{a}$, $\bm{b}$ and $(x_n)_{n=1}^{\infty}$ as in \cref{prethm:ErdosLemma4prime}.
We apply \cref{thm:criterion_irrationality_rational}.
Define sequences $(u_n)_{n=1}^{\infty}$ and $(z_n)_{n=1}^{\infty}$ of real numbers $\ge1$ by
\begin{equation}
\label{prethm:ErdosLemma4prime:def_un}
u_{n}
\coloneqq
\frac{x_{n}}{\max(1,\#\mathcal{N}_{\bm{a}}(x_{n}),\#\mathcal{N}_{\bm{b}}(x_{n}))}
\quad\text{for $n\ge 1$}
\end{equation}
and
\begin{equation}\label{prethm:ErdosLemma4prime:def_zn}
z_{n}\coloneqq u_{n}^{\frac{1}{2}}
\quad\text{for $n\ge 1$}.
\end{equation}
Clearly, $z_n,u_n\to\infty$ as $n\to\infty$
by condition \cref{prethm:ErdosLemma4prime:sparsity} in \cref{prethm:ErdosLemma4prime}. 
Hence, by \cref{prethm:ErdosLemma4prime:def_un} and \cref{prethm:ErdosLemma4prime:def_zn},
we have 
\begin{equation}\label{prethm:ErdosLemma4prime:strong_sparsity}
\#\mathcal{N}_{\bm{a}}(x_{n}),
\#\mathcal{N}_{\bm{b}}(x_{n})
\le
{x_{n}}/{u_{n}}
=
o({x_{n}}/{z_{n}})
\quad\text{as $n\to\infty$}.
\end{equation}
Moreover, by condition~\cref{prethm:ErdosLemma4prime:moderate_growth}
in \cref{prethm:ErdosLemma4prime}, we have
\begin{equation}
\label{prethm:ErdosLemma4prime:moderate_growth_modified}
S_{\bm{a}}(x_{n}),
S_{\bm{b}}(x_{n})
=
O(x_{n})
=
o(t^{z_{n}}x_{n})
\quad\text{as $n\to\infty$}.
\end{equation}
Thus, all conditions in \cref{thm:criterion_irrationality_rational} are satisfied from
\cref{prethm:ErdosLemma4prime:strong_sparsity},
\cref{prethm:ErdosLemma4prime:moderate_growth_modified}
and the assumptions of \cref{prethm:ErdosLemma4prime},
and so \cref{prethm:ErdosLemma4prime} follows
from \cref{thm:criterion_irrationality_rational}. 
\end{proof}

Next we show \cref{cor:chimera}. 
Let $\sigma(n)$ and $\phi(n)$ be the sum-of-divisors function and Euler's totient function,
respectively, as defined in \cref{sec:introduction}. 
We need some preparation.
\begin{lemma}[cf. {\cite[Theorem~2.9, p.~55]{MV:text}}] 
\label{lem:phi_sigma_bound}
We have 
\[
\phi(n)
\ge
\frac{n}{\log\log n}
\biggl( 
e^{-\gamma}+O\biggl(\frac{1}{\log\log n}\biggr)
\biggr)
\quad\text{as $n\to\infty$},
\]
where $\gamma=0.577215\dots$ is the Euler--Mascheroni constant. 
\end{lemma}

Let $x>1$ be large and $n\ge 3$ be an integer satisfying $\phi(n)\le x$. 
Then, by \cref{lem:phi_sigma_bound}, we have
\[
n
\ll
\phi(n)\log\log n
\ll
x\log\log n
\ll
x\sqrt{n},
\]
where, in this section, the implicit constants are independent of $n$ and $x$.
Hence, $\sqrt{n}\ll x$, so
\begin{equation}\label{eqphi1}
n\ll x\log\log x.
\end{equation}
The bound \cref{eqphi1} is also valid for $n=1,2$. 
We will use \cref{eqphi1} for the proof of \cref{cor:chimera}. 

For an arithmetic function $f\colon\mathbb{Z}_{\ge1}\to\mathbb{Z}_{\ge1}$ and $x>1$, we define 
\begin{equation}
\label{def:V}
\mathcal{S}_{f}
\coloneqq
\{f(n)\mid n\ge 1\}
\quad\text{and}\quad 
\mathcal{S}_{f}(x)
\coloneqq
\mathcal{S}_{f}\cap[1,x).
\end{equation}
For example, $\mathcal{S}_{\phi}$ denotes the set of values of the Euler's totient function
and $\mathcal{S}_{\phi}(10)=\{1,2,4,6,8\}$.

For the next result on the distribution of the image of $\phi$ and $\sigma$,
see e.g.\ Theorem on p.~264 and the last paragraph of Section~5
of Maier--Pomerance~\cite{MaierPomerance:PhiImage}.
For its improvement, see Ford~\cite{Ford:PhiImage}.

\begin{lemma}[Maier--Pomerance~\cite{MaierPomerance:PhiImage}]
\label{lem:V_f_bound}
As $x\to\infty$, we have
\[
\#\mathcal{S}_{\phi}(x),
\#\mathcal{S}_{\sigma}(x)
\le
\frac{x}{\log x}\exp\biggl(C(\log\log\log x)^{2}\biggr)
\]
with some absolute constant $C>0$.
\end{lemma}

\begin{proof}[Proof of \cref{cor:chimera}]
Let $t$ and $\bm{f}$ be as in \cref{cor:chimera} and $g$ be either $\phi$ or $\sigma$. 
Then, we have 
\begin{equation}
\label{cor:chimera:rewrite_series}
\sum_{m=1}^{\infty}
\frac{f(m)}{t^{g(m)}}
=
\sum_{n=1}^{\infty}
\sum_{\substack{m\geq1\\g(m)=n}}
\frac{f(m)}{t^{g(m)}}
=
\sum_{n=1}^{\infty}
\frac{a(n)}{t^{n}},
\end{equation}
where 
\begin{equation}
\label{cor:chimera:def_a}
a(n)
\coloneqq
\sum_{\substack{
m\ge 1\\
g(m)=n
}}
f(m)
\quad\text{for $n\ge1$}.
\end{equation}
We show the irrationality of \cref{cor:chimera:rewrite_series}
by applying \cref{thm:criterion_irrationality_rational}
to the above sequence $(a(n))_{n=1}^{\infty}$ and $b(n)\coloneqq 0$ for $n\ge 1$. 
Since $f(n)\ge 0$ for $n\ge 1$ and $\#\mathcal{N}_{\bm{f}}=\infty$, 
we have $a(n)\ge 0$ for $n\ge 1$ and $\#\mathcal{N}_{\bm{a}}=\infty$. 
Let $\delta>0$ be as in \cref{cor:chimera} and
\[
\mu\coloneqq\frac{2+\delta}{\log t}.
\] 
Define two sequences $(x_n)_{n=1}^{\infty}$ and $(z_n)_{n=1}^{\infty}$, respectively, by 
\begin{equation}
\label{cor:chimera:H_def}
x_n\coloneqq n
\quad\text{for $n\ge 1$}
\quad\text{and}\quad  
z_n
\coloneqq
\left\{
\begin{array}{cl}
1&\text{for $1\le n<e^{t}$},\\
\mu\log\log n&\text{for $n>e^{t}$}. 
\end{array}
\right.
\end{equation}
Let $n$ be a large integer and $m\geq1$ be an integer satisfying $g(m)\le n$. 
If $g=\sigma$, then $m\le\sigma(m)\le n$. 
In the case of $g=\phi$, we have \cref{eqphi1}.
Thus, for both of the cases $g=\sigma$ and $g=\phi$, we have 
\begin{equation}\label{eqphi2}
m<n\log n.
\end{equation}
Hence, by using
\cref{cor:chimera:def_a},
\cref{eqphi2} and the assumption
\cref{cor:chimera:cond}, we obtain 
\[
S_{\bm{a}}(x_n)
=
\sum_{m< x_n}
a(m)
=
\sum_{\substack{
m\ge1\\
g(m)\le x_n
}}
f(m)
\le
\sum_{m\le n\log n}
f(m)
=
O(n(\log n)^{1+\delta})
=
o(t^{z_{n}}x_n)
\quad\text{as $n\to\infty$},
\]
and $\varlimsup_{n\to\infty}\,a(n)^{1/n}=1$ since $a(n)\ge1$ for infinitely many integers $n$ and 
\[
a(n)
\le
S_{\bm{a}}(x_{n+1})
=
O(n(\log n)^{1+\delta})
\quad\text{as $n\to\infty$}.
\]
Moreover, by \cref{cor:chimera:def_a} and \cref{lem:V_f_bound}, we have 
\begin{align}
\#\mathcal{N}_{\bm{a}}(x_n)
=
\#\{
m\in[1,x_{n})\cap\mathbb{Z}
\mid
a(m)>0
\}
\le
\#\mathcal{S}_{g}(n)
&\le
\frac{n}{\log n}\exp\bigl(C(\log\log\log n)^{2}\bigr)\\
&=
o\bigl({x_{n}}/{z_n}) \quad\text{as $n\to\infty$}.
\end{align}
Thus, applying \cref{thm:criterion_irrationality_rational} yields \cref{cor:chimera}. 
\end{proof}

\bibliographystyle{amsplain} 
\bibliography{SparsePowerSeries}

\end{document}